\documentclass[twosided, a4paper, 11.5pt, reqno]{amsart}

\usepackage{amsfonts}
\usepackage{amssymb}
\usepackage{latexsym}
\usepackage{graphics}
\usepackage[2emode]{psfrag}
\usepackage{mathrsfs}
\usepackage{amsthm}
\usepackage{amsmath}
\usepackage[all]{xy}
\usepackage{enumerate}
\usepackage[latin1]{inputenc}
\usepackage{lscape}
\usepackage{a4wide}
\usepackage[usenames]{color}
\usepackage{multicol}
\usepackage{wasysym}
\usepackage{eufrak}
\usepackage{ifpdf}
\usepackage[bookmarks=false]{hyperref}

\usepackage{cancel}
\usepackage[normalem]{ulem}

\input{xy}
 \xyoption{all}
\newtheorem{theorem}{\sc Theorem}[section]
\newtheorem{proposition}[theorem]{\sc Proposition}

\newtheorem{lemma}[theorem]{\sc Lemma}
\newtheorem{corollary}[theorem]{\sc Corollary}
\theoremstyle{definition}
\newtheorem{definition}[theorem]{\sc Definition}

\newtheorem{example}[theorem]{\sc Example}
\newtheorem{examples}[theorem]{\sc Examples}

\theoremstyle{remark}
\newtheorem{remark}[theorem]{\sc Remark}

\newtheorem{corodefinition}[theorem]{\sc Corollary and Definition}

\newcommand{\tensor}[1]{\otimes_{\scriptscriptstyle{#1}}}

\newcommand{\Sf}[1]{\mathsf{#1}}
\newcommand{\fk}[1]{\mathfrak{#1}}
\newcommand{\cat}[1]{\mathcal{#1}}
\newcommand{\rmod}[1]{\Sf{Mod}_{#1}}

\renewcommand{\hom}[3]{\mathrm{Hom}_{#1}\left(#2,\,#3\right)}

\newcommand{\td}[1]{\widetilde{#1}}

\newcommand{\bara}[1]{\overline{#1}}

\newcommand{\lr}[1]{\left(\underset{}{} #1 \right)}
\newcommand{\Lr}[1]{\left[\underset{}{} #1 \right]}

\newcommand{\End}[2]{\mathrm{End}_{#1}(#2)}

\newcommand{\coring}[1]{\mathfrak{#1}}

\newcommand{\rcomod}[1]{ \mathsf{Comod}{}_{#1}}
\newcommand{\frcomod}[1]{ \mathsf{comod}{}_{#1}}

\newcommand{\LR}[1]{\left\{\underset{}{} #1 \right\}}

\newcommand{\I}{\mathbb{I}}

\newcommand{\B}[1]{\boldsymbol{#1}}

\newcommand{\Rep}[1]{\boldsymbol{\Rr}\mathbf{ep}_{\Bbbk}(\mathscr{#1})}
\newcommand{\REP}[1]{\boldsymbol{\Rr}\mathbf{ep}_{\Bbbk}(#1)}
\newcommand{\tRep}[1]{\boldsymbol{\Rr}\mathbf{ep}_{\Bbbk}^{\scriptscriptstyle{top}}(\mathscr{#1})}
\newcommand{\Repf}[1]{\rR_{\Bbbk}(\mathscr{#1})}
\newcommand{\RepF}[1]{\rR_{\Bbbk}(#1)}
\newcommand{\tRepf}[1]{\rR_{\Bbbk}^{\scriptscriptstyle{top}}(\mathscr{#1})}
\newcommand{\algb}[1]{\mathrm{M}_{\Bbbk}(#1_{\scriptscriptstyle{0}})}
\newcommand{\algt}[1]{\mathrm{M}_{\Bbbk}(#1_{\scriptscriptstyle{1}})}
\newcommand{\calgb}[1]{\mathrm{C}_{\Bbbk}(#1_{\scriptscriptstyle{0}})}
\newcommand{\calgt}[1]{\mathrm{C}_{\Bbbk}(#1_{\scriptscriptstyle{1}})}
\newcommand{\grpd}[1]{\xymatrix{#1_{1}\ar@<1.2ex>@{->}|-{\scriptstyle{\Sf{s}}}[r] \ar@<-1.2ex>@{->}|-{\scriptstyle{\Sf{t}}}[r] & \ar@{->}|-{\scriptstyle{\iota}}[l] #1_{0} }}
\newcommand{\proj}[1]{\mathsf{proj}(\mathrm{M}_{\Bbbk}(#1_{\scriptscriptstyle{0}}))}
\newcommand{\cproj}[1]{\mathsf{proj}(\mathrm{C}_{\Bbbk}(#1_{\scriptscriptstyle{0}}))}
\newcommand{\Gama}[1]{\B{\Gamma}(#1)}
\newcommand{\thom}[2]{\mathrm{T}_{\scriptscriptstyle{{#1}},\,\scriptscriptstyle{{#2}}}}
\newcommand{\tend}[1]{\mathrm{T}_{\scriptscriptstyle{{#1}}}}
\newcommand{\Grpd}{\mathsf{Grpd}}
\newcommand{\CHAlgd}{\mathsf{CHAlgd}}
\newcommand{\GTCHAlgd}{\mathsf{GTCHAlgd}}
\newcommand{\chara}{\mathscr{X}_{\Bbbk}}
\newcommand{\Alg}[2]{#1(#2)}
\newcommand{\Tanna}{\Sf{Tanna}_{\Bbbk}}
\newcommand{\Coring}[1]{#1\text{-}\Sf{Corings}}

\newcommand{\fF}{\mathscr{F}}
\newcommand{\gG}{\mathscr{G}}
\newcommand{\hH}{\mathscr{H}}

\newcommand{\jJ}{\mathscr{J}}

\newcommand{\lL}{\mathscr{L}}

\newcommand{\pP}{\mathscr{P}}

\newcommand{\rR}{\mathscr{R}}
\newcommand{\sS}{\mathscr{S}}

\newcommand{\vV}{\mathscr{V}}

\newcommand{\Aa}{\mathcal{A}}
\newcommand{\Bb}{\mathcal{B}}

\newcommand{\Ee}{\mathcal{E}}
\newcommand{\Ff}{\mathcal{F}}

\newcommand{\Hh}{\mathcal{H}}
\newcommand{\Ii}{\mathcal{I}}

\newcommand{\Kk}{\mathcal{K}}

\newcommand{\Pp}{\mathcal{P}}

\newcommand{\Rr}{\mathcal{R}}

\newcommand{\Tt}{\mathcal{T}}
\newcommand{\Uu}{\mathcal{U}}
\newcommand{\Vv}{\mathcal{V}}

\begin{document}
\allowdisplaybreaks

\title[Representative functions on groupoids and duality with Hopf algebroids]{Representative functions on discrete groupoids  and duality with Hopf algebroids.}
\author{Laiachi El Kaoutit}
\address{Universidad de Granada, Departamento de \'{A}lgebra. Campus De Ceuta. Facultad de Educaci\'{o}n y Humanidades.
Cortadura Del Valle s/n. E-51001 Ceuta, Spain}
\email{kaoutit@ugr.es}
\urladdr{http://www.ugr.es/~kaoutit}
\date{\today}
\thanks{Partially supported by grants  MTM2010-20940-C02-01 from the Ministerio de Educaci\'{o}n y Ciencia of Spain and  FQM-266, P11-FQM-7156  from Junta de Andalucía.}

\begin{abstract}
The aim of this paper is to establish a duality between the category of discrete groupoids and the category of geometrically transitive  commutative  Hopf algebroids in the sense of P. Deligne and A. Bruguières. In one direction we have the usual  contravariant functor which assigns to each Hopf algebroid its characters groupoid (the fiber groupoid at the ground field).  In the other direction we construct the contravariant functor which associated to each discrete groupoid its Hopf algebroid of representative functions.  
This duality  extends the well known duality between discrete groups and commutative Hopf algebras,  and also sheds light on a new  approach to Tannaka-Krein duality for compact topological groupoids. Our results are supported by several illustrative examples including  topological ones.
\end{abstract}

\maketitle

\pagestyle{headings}

\section*{Introduction}
The notion of smooth representative functions on a Lie group is a  powerful tool  whose usefulness is too remarkable  in the theory of Lie groups.
One of the aspects, perhaps the most impressive, where this notion  is evidently  indispensable,   is  Chevalley's reformulation of Tannaka's duality theorem  which is itself a  generalization of the classical Pontryagin duality theorem. Namely, it was shown by  Chevalley in \cite[Theorem 5, page 211]{Chevalley:book} that Tannaka's duality theorem for a compact Lie group  can be formulated as a theorem on the  characters group of its Hopf algebra of (complex valued) representative functions. This was a very elegant way to provide  any compact Lie group with a structure of algebraic real linear group.

These in fact are particular aspects  of a general theory  on compact topological groups and  more general theory on discrete groups. Specifically, it is well known from Hochschild's result \cite[Theorem 3.5, page 30]{Hochschild:book1} that the  contravariant functor which assigns to each compact topological group its algebra of (real valued) continuous representative functions, establishes an anti-equivalence between  the category of compact topological groups and  the category of commutative real Hopf algebras with gauge (a Hopf integral coming from the Haar measure) and with  dense characters group in the linear dual. When restrict to  finitely generated real Hopf algebras one obtains an anti-equivalence with the category of compact Lie groups.  Since any real Hopf algebra is a filtrated limit of its finitely generated Hopf subalgebras, one deduces from the previous anti-equivalence that any compact topological group is a projective limit of compact Lie groups, and so it is isomorphic to a  closed subgroup of a product of algebraic real linear groups.

In the more abstract case of discrete groups, one only  have a duality between groups and commutative Hopf algebras. Precisely, denote by $\mathsf{Grp}$ the category of groups and by $\mathsf{CHAlg}_{\Bbbk}$ the category of commutative Hopf algebras over a ground field $\Bbbk$.  Consider the contravariant functors $\rR_{\Bbbk}:\mathsf{Grp} \longrightarrow \mathsf{CHAlg}_{\Bbbk}$ (sending any group to its $\Bbbk$-valued representative functions) and $\chi_{\Bbbk}:  \mathsf{CHAlg}_{\Bbbk}\longrightarrow \mathsf{Grp}$ (associating to  any Hopf algebra its characters group). Then the pair $(\rR_{\Bbbk}, \chi_{\Bbbk})$ establishes a duality between   the categories $\mathsf{Grp}$ and $\mathsf{CHAlg}_{\Bbbk}$. This means that  there are natural transformations $id_{\scriptscriptstyle{\mathsf{Grp}}} \longrightarrow \chi_{\Bbbk} \rR_{\Bbbk}$ and $id_{\scriptscriptstyle{\mathsf{CHAlg}_{\Bbbk}}} \longrightarrow \rR_{\Bbbk} \chi_{\Bbbk}$ satisfying the usual triangles.  In other words, we have a natural isomorphism 
$
{\rm Hom}_{\scriptstyle{\Sf{CHAlg}_{\Bbbk}}}\big(- ;\rR_{\Bbbk}(+) \big)\, \cong \,  {\rm Hom}_{\scriptstyle{\Sf{Grp}}}\big(+; \chi_{\Bbbk}(-)\big)
$.

With the pertinent notations, we can see that all the above aspects are then captured in the following diagram:
\begin{equation*} 
\xymatrix@R=30pt{ \scriptstyle{\mathsf{Grp}}  \ar@/^1.1pc/@{->}|-{\scriptscriptstyle{\rR_{\Bbbk}}}[rr] \ar@{}|-{\scriptscriptstyle{\text{duality}}}[rr] && \scriptstyle{\mathsf{CHAlg}_{\Bbbk}}  \ar@/^1.1pc/@{->}|-{\scriptscriptstyle{\chi_{\Bbbk}}}[ll]  \\  \scriptstyle{\mathsf{CTGrp} } \ar@/^1.1pc/@{->}[rr] \ar@{}|-{\scriptscriptstyle{\text{anti-equivalence}}}[rr] \ar@{_{(}->}[u] &&  \ar@/^1.1pc/@{->}[ll] \scriptstyle{\td{\mathsf{CHAlg}}_{\mathbb{R}}} \ar@{^{(}->}|-{\scriptscriptstyle{\Bbbk=\mathbb{R}}}[u] \\ \scriptstyle{\mathsf{CLGrp}}  \ar@/^1.1pc/@{->}[rr] \ar@{}|-{\scriptscriptstyle{\text{anti-equivalence}}}[rr] \ar@{_{(}->}[u] &&  \ar@/^1.1pc/@{->}[ll] \scriptstyle{\td{\mathsf{CHalg}}_{\mathbb{R}}} \ar@{^{(}->}[u]  }\vspace{0,2cm}
\end{equation*}
which in fact provides a kind of dictionary between the geometric properties of groups and algebraic properties of theirs associated commutative Hopf algebras.  For instance,  given a compact topological group $G$  and takes the complex numbers as a base field, then  there is a Galois correspondence between closed subgroups of $G$ and  subalgebras of $\rR_{\mathbb{C}}(G)$ that contain the constants and are stable under conjugation and translations, see \cite[\S 5]{Hochschild/Mostow:1957}.  \smallskip

As was expounded in \cite{Brown:1987, Weinstein:1996}, the theory of groupoids in all its facets is,  in some sense, a modern replacement of the theory of groups, which  in fact supplies new techniques and tools for studying the symmetry and the structure of fairly complicated objects whose analysis through groups theory does not provide enough information.  Under this point of view, a modern replacement  of the algebraic contrapart of groups, i.e., commutative Hopf algebras, should be then  commutative Hopf algebroids. In this way the above functor of characters group,  should be then substituted by the  functor of \emph{characters groupoid}, which assigns to each Hopf algebroid its characters groupoid (i.e., the fiber groupoid at the ground field).

Employing  the functor of characters groupoid, it is then natural and reasonable to propose a suitable generalization of the duality and anti-equivalences in the above diagram, to the contexts of  groupoids and commutative Hopf algebroids.  Thus, 
in the left hand-side column of an analogue diagram for groupoids, it is clear what the categories should be: discrete groupoids, compact topological groupoids and compact Lie groupoids.
The main problem then lies on finding the adequate full subcategories of the category of commutative Hopf algebroids, in the right hand-side column, over which the functor of characters establishes a duality or an anti-equivalence (depending on the codomain category of groupoids).   In the case of compact  topological groupoids and specially for compact Lie groupoid, the search of such a  subcategory is  up to now more truncated and too difficult. Nevertheless, we solve hereby this problem for discrete groupoids.\smallskip

The chief concerns of this paper is then to establish  a duality between the category of  discrete groupoids $\Grpd$ and a full subcategory of commutative Hopf algebroids. Namely, the category of geometrically transitive Hopf algebroids  $\Sf{GTCHAlgd}_{\Bbbk}$, see the forthcoming subsection for definitions. Explicitly, our main theorem sated below as Theorem \ref{thm:main} which we quote here, says:

{\renewcommand{\thetheorem}{\bf{I}}
\begin{theorem}\label{thm:mainA}
Let $\Bbbk$ be a ground field. Then the contravariant functors of $\Bbbk$-characters groupoid and $\Bbbk$-valued representative functions 
$$\xymatrix{\chara: \Sf{GTCHAlgd}_{\Bbbk} \ar@<-0.5ex>@{->}[r] & \Grpd: \rR_{\Bbbk} \ar@<-0.5ex>@{->}[l]  }$$ 
establish a duality between the category of  geometrically transitive commutative Hopf algebroids and the category of discrete groupoids. That is, there is a natural isomorphism:
$$
{\rm Hom}_{\scriptstyle{\Sf{GTCHAlgd}_{\Bbbk}}}\lr{{(R,\Hh)}\, ;\,{\big(\algb{\gG},\Repf{G}\big)} }\,\, \cong \,\,  {\rm Hom}_{\scriptstyle{\Grpd}}\lr{{\gG}\, ; \,{\chara(R,\Hh)}}
$$
for every $(R,\Hh) \in \Sf{GTCHAlgds}_{\Bbbk}$ and $\gG \in \Grpd$ with  base algebra $\algb{\gG}={\rm Maps}(\gG_0,\Bbbk)$.
\end{theorem}
}
The case of discrete groups, that is, the duality in the top of the above diagram,  is then evidently recovered from Theorem {\ref{thm:mainA}}, since we already know that any commutative Hopf $\Bbbk$-algebra  is  geometrically transitive.  

Our techniques of proofs realise heavily on the theory of   Tannakian $\Bbbk$-linear categories. Namely, we first show in Proposition \ref{prop:tannaka} (perhaps a well known result) that  the category of representations of a  given discrete groupoid in finite dimensional $\Bbbk$-vector spaces is  in fact a  Tannakian $\Bbbk$-linear category, see the Appendix where we collect  the essential  definitions and basic properties of this notion.   The resulting commutative Hopf algebroid from the so called the reconstruction process,  is what hereby  refereed to as  \emph{the algebra of representative functions} on discrete groupoid, Corollary and Definition \ref{coro:comatrix}. The terminology is justified by showing that this algebra is in fact isomorphic to a subalgebra of the total algebra  attached to this groupoid, Proposition \ref{prop:zeta}.

This is mostly done in Sections \ref{sec:1} and \ref{sec:2}, where we also treat the case of topological groupoids with compact Hausdorff base space,  by showing  how to construct its Hopf algebroids of continuous representative functions, Remark \ref{rem:RFTops}. Contrary to  the discrete case, the (geometric) nature of this Hopf algebroid still unknown, at least for us. Nevertheless, we think that its construction  surely is  the  first step through a   correct way  to establish  an anti-equivalence between compact topological groupoids and  certain full subcatgeory of commutative Hopf algebroids, Remark  \ref{rem:final}.

Several examples were expounded along the paper, highlighting the case of an algebraic action $\Bbbk$-groupoids ($\Bbbk$ is an infinite field), where we were able to relate the \emph{Equivariant Serre Problem} concerning the triviality of algebraic equivariant bundles \cite{Masuda/Petrie:1995}, with some problems of its Hopf algebroid of polynomial representative functions, see Example \ref{exam:algGroups} (and also Remark \ref{rem:vb}).

\subsection*{Basic  notions and notations.}\label{ssec:0}

We fix a ground field $\Bbbk$, all rings are considered to be $\Bbbk$-algebras. The set of all $\Bbbk$-algebra maps $R \to S$ is denoted by $\Alg{R}{S}$. Since we are interested in (small) groupoids  with  set of objects in some cases is of the form $\Alg{R}{\Bbbk}$ for a commutative $\Bbbk$-algebra $R$. It is reasonable to assume that the set of object is a non empty set and hence to assume that all handled  commutative $\Bbbk$-algebras $R$   have the property that $\Alg{R}{\Bbbk} \neq \emptyset$. 

Let $\theta: R \to S$ be a morphism of commutative $\Bbbk$-algebras. We denote   
by $\theta^*: \rmod{R} \to \rmod{S}$ the extension functor, i.e., $\theta^*(-)=-\tensor{R}S$ and by $\theta_{*}: \rmod{S}\to \rmod{R}$ the restriction functor. 

All Hopf algebroids which will be considered here are commutative. Recall from \cite[A.I]{Ravenel:1986} that a \emph{commutative Hopf algebroid} is a pair of commutative $\Bbbk$-algebras $(R,\Hh)$  with $R\neq 0$ and $\Bbbk$-algebra maps 
$$\xymatrix@C=50pt{R\ar@<1.2ex>@{->}|-{\scriptstyle{\Sf{s}}}[r] \ar@<-1.2ex>@{->}|-{\scriptstyle{\Sf{t}}}[r] & \ar@{->}|-{ \scriptstyle{\varepsilon} }[l] \Hh}$$
where $\Sf{s}$ is the \emph{source} map, $\Sf{t}$ is the \emph{target} and $\varepsilon$ is the counit. Together with a \emph{comultiplication} $\Delta: \Hh \to \Hh\tensor{R}\Hh$, where $\Hh$ is considered as an $R$-bimodule with $\Sf{t}$ acting on the left and $\Sf{s}$ on the right, as well as an \emph{antipode} $\sS: \Hh \to \Hh$.  All these maps are required to satisfy the usual compatibilities conditions in the Hopf context (dual in some sense to a groupoid axioms). In this way Hopf algebroids are then presheaves of groupoids on affine schemes.  Under this point of view, there is a  relation between the category of comodules and the category of quasi-coherent sheaves with a groupoid action,  as was detailed in \cite{Hovey:2002} and \cite{Deligne:1990}.
 
For a given Hopf algebroid $(R,\Hh)$ we denote by $\rcomod{\Hh}$ its category of right comodules and by $\Uu_{\Hh}: \rcomod{\Hh} \to \rmod{R}$ the forgetful functor. This is  a symmetric  monoidal category with $\Uu_{\Hh}$ a strict monoidal functor. The  category of left $\Hh$-comodules is  monoidally isomorphic to  right $\Hh$-comodules, and so enjoys similar properties.
The full subcategory of $\rcomod{\Hh}$ whose objects are finitely generated as $R$-modules is denoted by $\frcomod{\Hh}$. It is well known that $\Hh_R$ (or ${}_R\Hh$) is flat iff $\rcomod{\Hh}$ is a Grothendieck category and $\Uu_{\Hh}$ is an  exact functor.  Under any one of these equivalent conditions $\Hh$ becomes  faithfully flat simultaneously as  left and right $R$-module.

A \emph{morphism of Hopf algebroids} is  a pair $\alpha=(\alpha_0,\alpha_1): (R,\Hh) \to (S,\Kk)$ of $\Bbbk$-algebra maps which are in a canonical way compatible with both Hopf algebroids structures. The \emph{extended morphism} of $\alpha$, is given by  $(id_S, \td{\alpha_1}): (S,S\tensor{R}\Hh\tensor{R}S) \to (S,\Kk)$ sending $p\tensor{R}h\tensor{R}q \mapsto \Sf{s}(p)\alpha_1(h)\Sf{t}(q)$, where $(S,S\tensor{R}\Hh\tensor{R}S)$ is endowed within a canonical structure of Hopf algebroid. 

The morphism $\alpha$ clearly induces a functor called \emph{the induction functor} $\alpha^*:= -\tensor{R}S: \rcomod{\Hh} \to \rcomod{\Kk}$ such that $\Uu_{\Kk} \circ \alpha^* = \alpha_0^* \circ \Uu_{\Hh}$. In other words, the functor $\alpha^*$ is a \emph{lifted functor}  of the functor $\alpha_0^*$. Whether in general a functor  $F:\rcomod{\Hh} \to \rcomod{\Kk}$ is lifted from some functor $\theta^*$, where $\theta: R \to S$ is an algebra map, is a quite  interesting question  from the Morita theory point of view, chiefly when $F$ is asked to be monoidal.  

In this paper we will often use \emph{geometrically transitive Hopf algebroids}, a notion which was introduced by A. Bruguières in \cite{Bruguieres:1994}, see also \cite{Deligne:1990}.
For sake of completeness and audience convenience, we include here the definition of this notion.  Although, as was mentioned above  we will restrict our self here to the commutative case. Recall from \cite[\S 5 Definition in page 5838]{Bruguieres:1994} that a given commutative Hopf algebroid $(R,\Hh)$ with a base ring $R$,  is said to be \emph{semi-transitive} if the category of (say right) comodules over the underlying $R$-coring $\Hh$ satisfies the following three conditions:
\begin{enumerate}[{$\Sf{ST}$}1]
\item Each object in $\frcomod{\Hh}$ is projective as $R$-module.
\item Each object in $\rcomod{\Hh}$ is a filtrate  inductive limit of objects in $\frcomod{\Hh}$.
\item The $\Bbbk$-linear category $\frcomod{\Hh}$ is  locally of finite type.
\end{enumerate}

The Hopf algebroid $(R,\Hh)$ is said to be \emph{geometrically semi-transitive} if $\Hh$ is projective as an $(R\tensor{\Bbbk}R)$-module and satisfies condition $\Sf{ST}$3. By \cite[Proposition 6.2]{Bruguieres:1994}, each geometrically semi-transitive $R$-coring is semi-transitive. In particular \cite[Proposition 6.2(i)]{Bruguieres:1994}
implies that $\Hh_R$ is a projective module whenever  so is the module $\Hh_{R\tensor{\Bbbk}R}$. 
Lastly, $(R,\Hh)$ is said to be \emph{geometrically transitive}\footnote{As was shown in \cite[Théorème 8.2]{Bruguieres:1994} this is equivalent to say that $(\Sf{Spec}(\Hh), \Sf{Spec}(R))$ is  a  transitive affine  $\Bbbk$-groupoid, in the sense of Deligne. That is, $(\Sf{s},\Sf{t}): \Sf{Spec}(\Hh) \to \Sf{Spec}(R) \underset{\scriptsize \Bbbk}{\times} \Sf{Spec}(R)$  is a cover in the \emph{fpqc} topology.}if it is geometrically semi-transitive and ${\rm End}_{\frcomod{\Hh}}(R)\cong\Bbbk$ where $R$ is considered as an $\Hh$-comodule  using the grouplike element $1_\Hh$\footnote{Semi-transitive Hopf algebroid with this property is called \emph{transitive}.}. 
Notice that condition $\Sf{ST}$2 follows   from the assumption  that $\Hh_R$ (or ${}_R\Hh$) is a projective module, either by using the theory of rational modules over the convolution $R$-algebra of $\Hh$  or directly by applying  \cite[Proposition 3.3]{Bruguieres:1994}.

Summing up, the  result \cite[Proposition 7.3]{Bruguieres:1994} says that a commutative Hopf algebroid $(R,\Hh)$ is geometrically transitive iff $\Hh$ is projective and faithfully flat $(R\tensor{\Bbbk}R)$-module, iff $(R_{\scriptscriptstyle{\mathbb{K}}},\Hh_{\scriptscriptstyle{\mathbb{K}}})$ is transitive for any field extension $\Bbbk \to \mathbb{K}$, where $R_{\scriptscriptstyle{\mathbb{K}}}=R\tensor{\Bbbk}\mathbb{K}$ and $\Hh_{\scriptscriptstyle{\mathbb{K}}}=\Hh\tensor{\Bbbk}\mathbb{K}$. Finally, it is noteworthy to mention that it is implicitly shown in \cite[\S 8]{Bruguieres:1994} that in our case, i.e., for commutative Hopf algebroid, a necessary and sufficient condition for $(R,\Hh)$ to be geometrically transitive is to be faithfully flat over  $R\tensor{\Bbbk}R$.

\section{The category of representations of discrete groupoids.}\label{sec:1}
We will recall here  the definition of the category of representations of discrete groupoid in $\Bbbk$-vector spaces. This  is a formal adaptation of the category of sheaves over a presheaf of groupoids on schemes  \cite{Deligne:1990, Breen:1994,Hovey:2002} or that of the category of equivariant sheaves over a topological action groupoids \cite{Bernstein/Lunts:1994,Segal:1968}. In what follows all statements will be given for discrete groupoids. In the  topological case   we limits our self to make some remarks which illustrate where  the difficulty lies in this case and perhaps  offer a well understanding of this situation. 

Our goal is to show that the  category of representations of a  discrete  groupoid in finite dimensional $\Bbbk$-vector spaces, is a Tannakian $\Bbbk$-linear category, see the Appendix for definitions. For topological groupoids we will see that this category  is instate a pseudo-Tannakian $\Bbbk$-linear category.\smallskip

Recall that a groupoid $\gG$ is a small category where each arrow is an isomorphism. The source and target  will be denoted by $\Sf{s}, \Sf{t}: \gG_1 \to \gG_0$, and the  map which assigns to each object its identity arrow is denoted by $\iota: \gG_0 \to \gG_1$. The inverse map and the composition as well as the rest of the axioms are supposed to be  understood. In all what follows,  a groupoid $\gG$ is implicitly assumed  to have a non empty set objects  $\gG_0 \neq \emptyset$. A \emph{morphism} between two  groupoids is a functor between their underlying categories.

\begin{definition}\label{def:rep}
Let $\gG: \xymatrix{\gG_{1}\ar@<1.2ex>@{->}|-{\scriptstyle{\Sf{s}}}[r] \ar@<-1.2ex>@{->}|-{\scriptstyle{\Sf{t}}}[r] & \ar@{->}|-{ \scriptstyle{\iota} }[l] \gG_{0} }$  be a  groupoid.  A \emph{representation} of $\gG$  or  an \emph{$n$-dimensional $\gG$-representation} consists on the following data:
\begin{enumerate}[1)]
\item $\Ee=\bigcup_{x \in G_0}E_x$ a disjoint union of finite dimensional  $\Bbbk$-vector spaces $E_{x}$ such that there exists an $n$-dimensional $\Bbbk$-vector space $V$ and linear isomorphisms $\varphi_x:V \to E_x$, for every $x \in X$.\label{item:1}
\item A family of linear isomorphisms $\{\varrho^{\Ee}_g\}_{g \in \gG_1}$ given by $\varrho^{\Ee}_g: E_{\Sf{s}(g)} \to E_{\Sf{t}(g)}$, for every $g \in \gG_1$ satisfying the cocycle condition. That is,  for every $x \in \gG_0$ and every  pair of composable arrows $g,h \in \gG_1$, we have 
\begin{equation}\label{Eq:cocycle}
\varrho^{\Ee}_{\iota(x)}\,= \,id_{E_x}, \quad \varrho^{\Ee}_{gh}= \varrho^{\Ee}_g \circ \varrho^{\Ee}_h.
\end{equation}\label{item:2}
\end{enumerate} 

A $\gG$-representation will be  denoted by a pair $(\cat{E},\varrho^{\Ee})$ (the corresponding vector space $V$ is implicitly understood) or simply by $\varrho^{\Ee}$ if there is no danger of misunderstanding.

Given a family of finite dimensional $\Bbbk$-vector spaces $\{E_x\}_{x \in \gG_0}$ which satisfies condition  \ref{item:1}) in Definition \ref{def:rep}, we can consider the so called \emph{frame groupoid} $\Phi(\Ee)$ whose set of objects is $\gG_0$ and the set of arrows is the set of linear isomorphisms between the (fibers) $E_x$'s. In this way,  a  $\gG$-representation $(\Ee,\varrho^{\Ee})$ is nothing but a morphism of groupoids (functor) $\varrho: \gG \to \Phi(\Ee)$, where $\varrho(x)=x$ and $\varrho(g)=\varrho^{\Ee}_g$. 
\end{definition}
Take $\Ii=\bigcup_{x \in \gG_0}\Ii_x$, with $\Ii_x=\Bbbk$, for every $x \in \gG_0$ and $\varrho^{\Ii}_g=id_{\Bbbk}$, for every $g \in \gG_1$. Then obviously, $\Ii$ is a $1$-dimensional $\gG$-representation. The zero dimensional $\gG$-representation, is the representation with fibers the zero $\Bbbk$-vector space. 

Consider two $\gG$-representations $(\Ee,\varrho^{\Ee})$ and $(\Ff, \varrho^{\Ff})$. A \emph{morphism of representation} from $(\Ee,\varrho^{\Ee})$ to $(\Ff, \varrho^{\Ff})$ is a family of linear maps $\{\alpha_x\}_{x \in \gG_0}$, $\alpha_x: E_x \to F_x$  rendring commutative the following diagrams
\begin{equation}\label{Eq:diag}
\xymatrix@C=50pt{  E_{\Sf{s}(g)} \ar@{->}^-{\varrho^{\Ee}_g}[rr]  \ar@{->}_-{\alpha_{\Sf{s}(g)}}[d] & & E_{\Sf{t}(g)} \ar@{->}^-{\alpha_{\Sf{t}(g)}}[d] \\  F_{\Sf{s}(g)} \ar@{->}^-{\varrho^{\Ff}_g}[rr] & & F_{\Sf{t}(g)}  }
\end{equation}

We denote by $\Rep{G}$ \emph{the category of $\gG$-representations in $\Bbbk$-vector spaces}. The $\Bbbk$-vector space of morphisms in this category will be denoted as usual by $\hom{\Rep{G}}{\varrho^{\Ee}}{\varrho^{\Ff}}$. 
Before going on  given the properties of this category, let us discuss  the topological case and give some elementary examples.

\begin{remark}\label{rem:rep}
Up to now we could perfectly define the category of representations of topological groupoids  by taking the category of locally trivial  vector bundles over $\gG_0$. Precisely, given a topological groupoid $\gG$ (see for instance \cite[Definition 2.1.]{Renault:1980}), a $\gG$-representation  is a three-tuple $(\Ee,\pi_{\Ee}, \varrho^{\Ee})$ where $\pi_{\Ee}: \Ee \to \gG_0$ is a locally trivial real (or complex) vector bundle and $\varrho^{\Ee}: \Sf{s}^*\Ee \to \Sf{t}^*\Ee$ is an isomorphism  between the induced (or pull-back) vector bundles which satisfies the cocycle conditions. A morphism between two $\gG$-representations, is then  a morphism of vector bundles $\alpha: \Ee \to \Ff$  compatible with the actions, that is, satisfies $\Sf{t}^*\alpha \circ \varrho^{\Ff} \,=\, \varrho^{\Ee } \circ \Sf{s}^*\alpha$. 

This category is denoted by $\tRep{G}$, and clearly we have  a functor $\tRep{G} \to \Rep{G}$.  Of course if we consider our discrete groupoid as a discrete topological groupoids and the base field a discrete field (i.e.~each finite dimensional vector space is discrete), then clearly Definition \ref{def:rep} is a particular instance of this one. That is, the previous functor is an equality. 

There are also in the literature others specific notions of representations on some classes of topological groupoids. For example in the case of locally compact  or Lie groupoids, one can speaks about continuous representations, see  \cite{Westman:1967,Renault:1987,Amini:2007,Bos:2011}. 
\end{remark}

\begin{example}\label{exam:action}
Let $\gG$ be an \emph{action groupoid}\footnote{\emph{Semi-direct groupoid} in the terminology of \cite{Brown:1987}.}. This means that we are given a group $G$ and  say a left  $G$-set $X$ with action  $G\times X \to X$, $(g,x) \mapsto gx$, where we set  $\gG_1=G\times X$ and $\gG_0=X$ with $\Sf{s}(g,x)=x$, $\Sf{t}(g,x)=gx$ and $\iota(x)=(e,x)$ ($e$ is the neutral element of $G$). Thus a $\gG$-representation is a family of $n$-dimensional $\Bbbk$-vector space $\{E_x\}_{x \in X}$ (each of them is isomorphic to some fixed $n$-dimensional $\Bbbk$-vector space) and there is linear isomorphisms $\varrho^{\Ee}_g:E_x \to E_{gx}$, for every $g \in G$. In other words the fibers are 'equivariant' under the action of $G$. 

If $G$ is a topological group and $X$ is a topological space such that the $G$-action is a continuous map ($\Bbbk$ is the field of real or complex numbers). Then the category $\tRep{G}$  described in Remark \ref{rem:rep},  is nothing but the category of \emph{$G$-equivariant vector bundles} (or \emph{$G$-vector bundles}). In particular, if $X=\{*\}$ is a one point set, then the category $\tRep{G}$ coincides with category of continuous $G$-representations. Of course, when  $G$ is a discrete group, $X=\{*\}$ and $\Bbbk$ is any field, then $\B{\Rr {\rm ep}_{\Bbbk}}(G)$  obviously coincides with the classical category of $G$-representations in finite dimensional $\Bbbk$-vector spaces.
\end{example}

\begin{example}[The additive   and multiplicative  $\Bbbk$-groupoids $\gG^{\scriptscriptstyle{a, -}}$ and $\gG^{\scriptscriptstyle{m, -}}$]\label{exam:additive}
Fix a commutative $\Bbbk$-algebra $A$. Given another commutative $\Bbbk$-algebra $R$  we consider the following small category. The set of objects is $A(R)$ the set of all $\Bbbk$-algebras maps from $A$ to $R$. A morphism $f: \sigma \to \gamma$ between two objects $\sigma , \gamma \in A(R)$ is given by an element $r(f) \in R$. The identity arrow of an object $\sigma$ is given by the element $r(1_{\sigma})=0$. The composition is the sum $r(f\circ g)=r(f) + r(g)$, and each arrow is invertible by taking $r(f^{-1})=-r(f)$. We refer to this groupoid as \emph{the additive groupoid of $R$ over $A$} and denoted by $\gG^{\scriptscriptstyle{a, A}}(R)$. 
The \emph{multiplicative groupoid $\gG^{\scriptscriptstyle{m, A}}(R)$ of $R$ over $A$} is constructed  by taking the same set of object $A(R)$ where a morphism  $f: \sigma \to \gamma$ is  given in this case by an element $s(f) \in \Uu(R)$ of the unit group of $R$. The identity arrow of $\sigma$ is $s(1_\sigma)=1_R$ the identity element of $R$, the composition is the multiplication $s(f\circ g)=s(f)s(g)$, and each arrow $f$ have an inverse given by $s(f^{-1})=s(f)^{-1}$. 

It is easily seen that    $\gG^{\scriptscriptstyle{a, A}}$ and $\gG^{\scriptscriptstyle{m, A}}$ are affine  $\Bbbk$-groupoids, respectively, represented by the geometrically transitive Hopf algebroids $\big(A, (A\tensor{\Bbbk}A)[T]\big)$ and $\big(A, (A\tensor{\Bbbk}A)[T, T^{-1}]\big)$.  Obviously, when $A=\Bbbk$, one recover the notions of additive and multiplicative  affine $\Bbbk$-groups $G_a$ and $G_m$. 

Assume now that $R$ and $A$ are finite dimensional $\Bbbk$-algebras and consider the multiplicative groupoid  $\gG^{\scriptscriptstyle{m, A}}(R)$ of $R$ over $A$. For any object $\sigma \in A(R)$ we put $R_{\sigma}=R$ as a $\Bbbk$-vector space.  Take the 'vector bundle' $\Rr:=\cup_{\sigma \in R(A)} R_{\sigma}$ and for any arrow $f:\sigma \to \gamma$ in $\gG^{\scriptscriptstyle{m, A}}(R){}_1$, set $\varrho^{\Rr}_f: R_{\sigma} \to R_{\gamma}$ by sending $x \mapsto s(f)x$. Then it is clear that $(\Rr, \varrho^{\Rr})$ is a $\gG^{\scriptscriptstyle{m, A}}(R){}$-representation.
\end{example}

Now, we come back to the properties of the category of representations of a groupoid. Here are some  basic well known facts in this category.

\begin{lemma}\label{lema:rep}
Let $\gG$ be a groupoid. Then its category of representations $\Rep{G}$, is a $\Bbbk$-linear  abelian  symmetric rigid (or autonomous) monoidal category. 
\end{lemma}
\begin{proof}
The fact that $\Rep{G}$ is an abelian category, follows directly from $\Bbbk$-vector spaces. Thus, the construction of  kernels and cokernels as well as the reminder axioms of an abelian category can be checked fiberwise.   For instance, given a morphism $\alpha \in \Rep{G}$, then one can easily show that $\mathrm{Ker}(\alpha):= \cup_{x \in \gG_0}\mathrm{Ker}(\alpha_x)$ and  $\mathrm{Coker}(\alpha):= \cup_{x \in \gG_0}\mathrm{Coker}(\alpha_x)$ are the underlying sets, respectively, of  the kernel and cokernel of $\alpha$ in $\Rep{G}$. The rest of the statements  follows also from the category of $\Bbbk$-vector spaces. The tensor product and duals are given as follows:
\begin{equation}\label{Eq:otimes}
(\Ee,\varrho^{\Ee})\tensor{}(\Ff,\varrho^{\Ff}):=\,\, (\Ee\tensor{}\Ff,\varrho^{\Ee}\tensor{} \varrho^{\Ff})\,\,=\,\, \lr{\bigcup_{x \in \gG_0}E_x\tensor{\Bbbk}F_x, \{\varrho^{\Ee}_g\tensor{\Bbbk}\varrho^{\Ff}_g\}_{g \in \gG_1}}
\end{equation}
\begin{equation}\label{Eq:star}
(\Ee,\varrho^{\Ee})^*:=\,\, (\Ee^*, \varrho^{\Ee^*})\,\,=\,\, \lr{\bigcup_{x \in \gG_0}E_x{ }^*,\left\{\underset{}{}[(\varrho^{\Ee}_g)^*]^{-1}\right\}_{g \in \gG_1}  },
\end{equation}
where we denote by $X^*=\hom{\Bbbk}{X}{\Bbbk}$ the linear dual of a vector space $X$ and similar notation is used for linear maps. The identity object is the $1$-dimensional representation $(\Ii,\varrho^{\Ii})$.
\end{proof}

As in the case of groups or rings, morphisms between groupoids entail functors between their categories of representations. 

\begin{lemma}\label{lema:indfunct}
Let $\phi: \gG \to \hH$ be a morphism of groupoids. Then there is a $\Bbbk$-linear functor  called the restriction functor
\begin{equation}\label{Eq:ind}
\B{\Rr}(\phi): \Rep{H} \longrightarrow  \Rep{G}, 
\end{equation}
sending any representation $(\Pp,\varrho^{\Pp}) \in \Rep{H}$ to the representation $(\phi_0^*{\Pp},\phi_1^*\varrho^{\Pp})$, where $\phi_0^*{\Pp}\,=\, \bigcup_{x \in \gG_0}P_{\phi_0(x)}$ and $(\phi_1^*\varrho^{\Pp})_{g}\,=\, \varrho^{\Pp}_{\phi_1(g)}$, for every $x \in \gG_0$ and $g \in \gG_1$, and acting in a obvious way on morphisms. Moreover, the restriction  functor $\B{\Rr}(\phi)$ is  strong  monoidal.
\end{lemma}
\begin{proof}
Straightforward. 
\end{proof}

\begin{example}\label{exam:actionI}
Let $\gG$ be an action groupoid as in Example \ref{exam:action} with 
a group $G$ and  a left  $G$-set $X$. Consider $G$ as a groupoid with one object. Then the projection $pr_1: G \times X \to G$ defines a morphism of groupoids. The associated restriction functor $\B{\Rr}(pr)$  coincides then with the functor $-\times X$, which sends any $G$-representation $(V, \varrho^{\scriptscriptstyle{V}})$ to the $\gG$-representation 
$V\times X =\cup_{x \,\in X} V\times \{x\}$ with $\gG$-action $\varrho^{\scriptscriptstyle{V\times X}}_{(g, x)}: V\times\{x\} \to V\times\{gx\}$ sending $(v, x) \mapsto (\varrho^{\scriptscriptstyle{V}}_g(v), gx)$, for every $g \in G$. Of course, we have $\B{\Rr}(pr)(f)=f \times X$, for any morphism $f :(V, \varrho^{\scriptscriptstyle{V}}) \to (W, \varrho^{\scriptscriptstyle{W}})$ of $G$-representations.
\end{example}

\begin{remark}\label{remak:Frobenius}
When both $\gG$ and $\hH$ are groups in Lemma \ref{lema:indfunct}, we have (see Example \ref{exam:action}) that the restriction functor  coincides with the usual restriction functor in group theory. Obviously  this functor can be extended to all representation (including the infinite dimensional ones). The search of a (left) adjoint functor to $\B{\Rr}(\phi)$, that is,  the construction of the \emph{induction functor}, is a classical subject in group theory which is strongly related to \emph{Frobenius reciprocity formula}. In our case we think that it is possible to construct  \emph{the induction functor for groupoids}, i.e.~ an adjoint functor to $\B{\Rr}(\phi)$ of Lemma \ref{lema:indfunct}.  Naturally,  this construction could be very interesting in the study of discrete groupoids, as it was for discrete groups, however,  we think that this deserves a separate project.
\end{remark}

\begin{remark}\label{rem:vb}
If we take a topological groupoid $\gG$. Then it is well know that the category   $\tRep{G}$ of representations described in Remark \ref{rem:rep},  is no longer abelian. Although, it is  symmetric rigid monoidal $\Bbbk$-linear category. In the discrete case the advantage was  that for any  $\gG$-representation $(\Ee,\varrho^{\Ee})$ the underlying set $\Ee$ can be identified with $\gG_0\times V$ for some finite dimentional $\Bbbk$-vector space $V$ (i.e., it is a trivial bundle over $\gG_0$) and there were no topological properties need to be checked. In the topological case we can show, using \cite[Theorem 6.3]{Karoubi:book} and diagrams \eqref{Eq:diag}, that the category $\tRep{G}$  is in fact a \emph{pseudo-abelian category}  in the sense of M. Karoubi \cite[Definition 6.7]{Karoubi:book}.

On  the other hand, the functor of Lemma \ref{lema:indfunct} can be also constructed for  topological groupoids. Here of course we are implicitly assuming that the structure maps $\phi_0,\phi_1$ are continuous, although,  one only need  to assume that $\phi_1$ is continuous. In this case, the construction  of that functor  uses in part  the induced vector bundles by  $\phi_0$. 
For instance, in the case of a topological action groupoid Example \ref{exam:action}, we have as in Example \ref{exam:actionI} a morphism of topological groupoids $pr_1: \gG_1=G \times X \to G$ which leads to the functor $\B{\Rr}^{\scriptscriptstyle{top}}(pr): \B{\Rr{\rm ep}}_{\Bbbk}^{\scriptscriptstyle{top}}(G) \to \tRep{G}$ form the category of continuous $G$-representations to the category of $G$-equivariant bundles. In this direction, assume further that $X \in \B{\Rr{\rm ep}}_{\Bbbk}^{\scriptscriptstyle{top}}(G)$,  we can address here a similar problem to the \emph{Equivariant Serre Problem} \cite{Masuda/Petrie:1995} (see also  Example \ref{exam:algGroups} below). Precisely,   we can ask whether the functor $\B{\Rr}^{\scriptscriptstyle{top}}(pr)$ is 'surjective'  and  naturally isomorphic to  $\B{\Rr}^{\scriptscriptstyle{top}}(pr) \,\cong\, -\times X$. That is, one can ask  whether a  given $G$-equivariant bundle is trivial, i.e., isomorphic to $V\times X$ for some continuous $G$-representation $V$.
\end{remark}

The following notations will be used in the sequel.  Let $\gG$ be a groupoid, we denote by $\algb{\gG}$  and $\algt{\gG}$ the commutative $\Bbbk$-algebras of all maps, respectively,  from  $\gG_0$  and $\gG_1$ to $\Bbbk$. The algebra $\algb{\gG}$ is refereed to as \emph{the base algebra of $\gG$} and $\algt{\gG}$ \emph{the total algebra of $\gG$}. We consider $\algt{\gG}$ as an $(\algb{\gG}\tensor{\Bbbk}\algb{\gG})$-algebra via the algebra map 
$$
\Sf{t}^* \tensor{\Bbbk}\Sf{s}^*: \algb{\gG}\tensor{\Bbbk}\algb{\gG} \longrightarrow \algt{\gG},
$$ where  $\Sf{s}^*:=\mathrm{M}_{\Bbbk}(\Sf{s})$ and $\Sf{t}^*:=\mathrm{M}_{\Bbbk}(\Sf{t})$ are the corresponding $\Bbbk$-algebra maps of $\Sf{s}$ and $\Sf{t}$. 
We denote by $\proj{\gG}$ the category of finitely generated and projective $\algb{\gG}$-modules.

Take a $\gG$-representation $(\Ee,\varrho^{\Ee})$ and let $\pi_{\Ee}: \Ee \to \gG_0$ be the canonical surjection. We consider the $\algb{\gG}$-modules of ('global') sections. Precisely, we set
\begin{equation}\label{Eq:Gamma}
\Gama{\Ee}\,=\, \left\{ s: \gG_0 \to \Ee|\, \pi_{\Ee} \circ s \,=\, id_{\gG_0} \right\}.
\end{equation}
This is a $\Bbbk$-vector space endowed with structure of $\algb{\gG}$-modules explicitly given by:
$$
(s+s')(x)\,=\, s(x) + s'(x),\; (a.s)(x)\,=\, a(x) s(x), \text{ for every } x \in \gG_0, a \in \algb{\gG}, \text{ and }s,s' \in \Gama{\Ee}.
$$ 

In fact $\Gama{\Ee}$ is a finitely generated and projective $\algb{\gG}$-module. Namely, this can be checked either by thinking of $(\Ee, \pi_{\Ee})$ as a trivial bundle or by directly computing a dual basis. A precise  way of computing a dual basis is given as follows. 
Consider the $n$-dimensional underlying  vector space $V$  (the 'type fiber') together with the linear isomorphisms $\varphi_x: V \to E_x$. Fix, $\{v_1, \cdots, v_n\}$ a basis for $V$, and consider, for any $i=1,\cdots,n$ the following section $s_i:\gG_0 \to \Ee$ sending $x \mapsto \varphi_x(v_i)$. Since $\{\varphi_x(v_i)\}_{1,\cdots,n}$ is a basis for any of the $E_x$'s, so  for any section $s \in \Gama{\Ee}$, we can  write $s(x)\,=\, \sum_{i} a_i(x) s_i(x)$, for some $a_i(x) \in \Bbbk$ which depends only on $s(x)$\footnote{Notice here that $a_i(x)=v_i^* \circ \varphi_x^*$, where the $v_i^*$'s are the dual maps  of the basis $\{v_1,\cdots,v_n\}$.}. Thus, the maps $a_i$'s define an $\algb{\gG}$-linear maps $s_i^*: \Gama{\Ee} \to \algb{\gG}$ by sending $s \mapsto a_i$. Now, one can easily check that $\{s_i, s_i^*\}_i$ form a dual basis for the $\algb{\gG}$-module  $\Gama{\Ee}$.

Evidently,  after forgetting the $\gG$-action, $\B{\Gamma}$ establishes  a $\Bbbk$-linear functor from the category of representations  $\Rep{G}$ to the category of modules $\proj{\gG}$. Specifically, we have  the following well known result (see the last observation in  \cite[Example 1.24]{Deligne/Milne}).

\begin{proposition}\label{prop:tannaka}
Let $\gG$ be a groupoid. Then the functor 
$$
\xymatrix@R=0pt{\omega_{\gG}:  \Rep{G} \ar@{->}[rr] & &  \proj{\gG}\\  (\Ee,\varrho^{\Ee}) \ar@{->}[rr] & & \Gama{\Ee} }
$$
is a non trivial $\Bbbk$-linear monoidal right exact functor. In particular $\omega_{\gG}$ is a fiber functor and the pair $(\Rep{G},\omega_{\gG})$ is then a Tannakian $\Bbbk$-linear category.
\end{proposition} 
\begin{proof}
We refer to the Appendix for the definitions and terminology employed in the statement.  
Notice that in our case there are various ways to deduce that $\omega_{\gG}:=\omega$ is a  monoidal functor. Here we give an elementary proof.  First we clearly have that $\omega(\Ii) \cong \algb{\gG}$. Now,  take two $\gG$-representations $(\Ee,\varrho^{\Ee})$ and $(\Ff,\varrho^{\Ff})$, and consider the following $\Bbbk$-linear map
$$
\psi: \Gamma(\Ee)\tensor{\Bbbk}\Gamma(\Ff) \longrightarrow \Gamma(\Ee\tensor{}\Ff), \lr{ p\tensor{\Bbbk}q \longmapsto \LR{x \mapsto p(x)\tensor{\Bbbk}q(x)} }.
$$

This map can be clearly extended to an $\algb{\gG}$-linear map  on  the tensor product  $\Gamma(\Ee)\tensor{\algb{\gG}}\Gamma(\Ff)$. We denote also by $\psi$ this extension. It is easily checked that $\psi$ is injective. On the other hand, take an element $s \in \Gamma(\Ee\tensor{}\Ff)$, so for every $x \in \gG_0$, we can write $s(x)\,=\, \sum_{i,j}a_{i,j}(x)s_i(x)\tensor{\Bbbk}r_j(x)$,  where $\{s_i,s_i^*\}$ and $\{r_j,r_j^*\}$, are respectively, the ('global') dual basis of $\Gamma(\Ee)$ and $\Gamma(\Ff)$ constructed as above. We clearly have that $\psi\big(\sum_{i,j}s_i\tensor{\algb{\gG}}a_{i,j}r_j\big)=s$ and then $\psi$ is surjective. This establishes a natural isomorphism $\omega(-)\tensor{\algb{\gG}}\omega(-) \cong \omega(-\tensor{}-)$  obviously satisfying the necessary coherence conditions which convert $\omega$ into a monoidal functor. 

We need to check that $\omega$ is right exact. So take an epimorphism $\alpha: (\Ee, \varrho^{\Ee}) \to (\Ff,\varrho^{\Ff}) \to 0$ in $\Rep{G}$. Then $\omega_\alpha: \Gama{\Ee} \to \Gama{\Ff}$ is defined by sending any section $\Gama{\Ee} \ni s \mapsto \omega_\alpha(s) \in \Gama{\Ff}$ to the section defined by $\omega_\alpha(s)(x)=\alpha_x(s(x))$, for every $x \in \gG_0$. Given a section $r \in \Gama{\Ff}$, we know that, for every $x \in \gG_0$, there are elements $e^{(x)} \in E_x$ such that $\alpha_x(e^{(x)})=r(x)$. Fix an element $x \in \gG_0$ and choose only  one element say $e_x$ of those $e^{(x)} \in E_x$. In this way, we have define a section  $s: \gG_0 \to \Ee$ which sends $x \mapsto e_x$ and satisfies $\alpha_x(s(x))=\alpha_x(e_x)=r(x)$. That is, $\omega_\alpha(s)=r$, and so $\omega(\alpha)$ is an epimorphism\footnote{This step can be directly deduced  by applying Urysohn's lemma to  $\gG_0$ considering it as a  normal space with discrete topology.}.
The particular statement follows by Lemma \ref{lema:rep}, using  the fact that $\Bbbk\cong \End{\Rep{G}}{\varrho^{\Ii}}$ and by knowing that $\omega$ is a faithful functor.
\end{proof}

\begin{remark}\label{rem:NotTanaka}
For a topological groupoid $\gG$ with compact Hausdorff base space $\gG_0$ a variant of Proposition \ref{prop:tannaka} could be formulated. Explicitly, take the category $\tRep{G}$ of $\gG$-representations as was defined in Remark \ref{rem:rep} with $\Bbbk$ denotes the field of real or complex numbers.  As it was mentioned before,  this category is  a $\Bbbk$-linear pseudo-abelian symmetric rigid monoidal category with ${\rm End}(\Ii) \cong \Bbbk$.  A functor $\omega^{top}: \tRep{G} \to \rmod{\calgb{\gG}}$ can be naturally defined here by taking continuous global sections and the ring of continuous functions  $\calgb{\gG}$ as the base ring. Hence we can show thanks to \cite{Swan:1962,Serre:1958} that its image lands in  the sub-category $\cproj{\gG}$ of finitely generated and projective $\calgb{\gG}$-modules.
To show that $\omega^{top}$ is a monoidal functor,  one perhaps  could uses similar arguments  as in  the proof of Proposition \ref{prop:tannaka}, since we know  in this case that a 'global' dual basis always exits; in the sense that for any vector bundle $\Ee$ over $\gG_0$ (of locally constant rank)  there are continuous sections $s_1, \cdots, s_n \in \Gama{\Ee}$ such that for any $x \in\gG_0$ there is a neighborhood $U$ of $x$ where the set  $\{s_1(y), \cdots, s_n(y)\}$ is a $\Bbbk$-basis of $E_y$, for any $y \in U$.

On the other hand, since in general we know that the fiber functor on Tannakian $\Bbbk$-linear categories is in fact an exact functor, so it is left exact. We could then  expect in the topological case that the functor $\omega^{top}$ preserves the kernels of projectors\footnote{That are  endomorphisms $f$ such that $f^2=f$~.}. As conclusion, a topological version of Proposition \ref{prop:tannaka} could be formulated by saying that   $(\tRep{G}, \omega^{top})$ is a \emph{pseudo-Tannakian $\Bbbk$-linear category}.
\end{remark}
Recall that a $\Bbbk$-linear category is said to be  \emph{locally of finite type over $\Bbbk$}, if any object is of finite length and the $\Bbbk$-vector space of morphisms between arbitrary two objects is finite dimensional.  
\begin{corollary}\label{coro:1}
Let $\gG$ be a groupoid and consider $\Rep{G}$ its Tannakian $\Bbbk$-linear category of representations. Then  $\Rep{G}$ is locally of finite type over $\Bbbk$ and the canonical map $$\xymatrix@R=0pt{\hom{\Rep{G}}{\varrho^{\Ee}}{\varrho^{\Ff}}\tensor{\Bbbk}\algb{\gG} \ar@{->}[rr] &&  \hom{\algb{\gG}}{\Gama{\Ee}}{\Gama{\Ff}}  \\  \alpha\tensor{\Bbbk}a \ar@{|->}[rr]&& \left[\underset{}{}  s \longmapsto[ x \mapsto \alpha_x(s(x))a(x) ] \right] }$$ is injective.
\end{corollary}
\begin{proof}
This is a direct consequence of Proposition \ref{prop:tannaka} and \cite[Proposition 2.5]{Bruguieres:1994}.
\end{proof}

\section{The Hopf algebroid of representative functions.}\label{sec:2}
In this section we introduce the Hopf algebroid of representative functions on a given discrete groupoid. In the discrete group case, we recover the construction of the commutative Hopf algebra of representative functions \cite{Abe:book, Hochschild:book}, see  also \cite{Chevalley:book, Hochschild/Mostow:1957, Segal:1968a} for  topological groups, Lie groups and algebraic linear groups with infinite base field.  
\smallskip

By  groupoid we  mean a  discrete groupoid, except where  otherwise  it is explicitly specified.  The notation is that of Section \ref{sec:1}. We fix a groupoid $\gG: \grpd{\gG}$ and denote by $B:=\algb{\gG}$ its base $\Bbbk$-algebra. As was  mentioned before, we consider the total $\Bbbk$-algebra  $\algt{\gG}$ as an $(B\tensor{\Bbbk}B)$-algebra via the duals of both $\Sf{s}$ and $\Sf{t}$, i.e.~ $\mathrm{M}_{\Bbbk}(\Sf{s})$ and $\mathrm{M}_{\Bbbk}(\Sf{t})$.

By using Proposition \ref{prop:tannaka}, we can  apply 
the procedure described in Appendix to construct the universal object $\Repf{G}:=\lL_{\Bbbk}(\omega_{\gG})$, see 
\cite{Deligne:1990} or precisely \cite[Théorème 5.2]{Bruguieres:1994},   see also \cite[\S 4]{El Kaoutit/Gomez:2004}.

Let us first  fix some new notations.  For a  $\gG$-representation $(\Ee,\varrho^{\Ee})$ we  denote $\omega(\Ee,\varrho^{\Ee}):=\Gama{\Ee}$ its image by the fiber functor of Proposition \ref{prop:tannaka}. We will often use the isomorphisms $ \Gama{\Ee^*} \cong  \Gama{\Ee}^*=\hom{B}{\Gama{\Ee}}{B}$. By applying the description offered in equation \eqref{Eq:Lomega} (see Appendix) to the fiber functor $\omega_{\gG}: \Rep{G} \to \proj{\gG}$, we obtain

\begin{corodefinition} \label{coro:comatrix}
Let $\gG$ a groupoid and $B=\algb{\gG}$ its base  $\Bbbk$-algebra. Then the universal $B$-bimodule $\Repf{G}$ is the commutative $(B\tensor{\Bbbk}B)$-algebra 
\begin{equation}\label{Eq:RG}
\Repf{G} \,\,\cong\,\, \int^{\varrho^{\Ee}\,\in\, \Rep{G}} \Gama{\Ee}^*\tensor{\Bbbk}\Gama{\Ee}\, \,\, \cong\,\, \frac{\bigoplus_{\varrho^{\Ee}\, \in \,\Rep{G}} \Gama{\Ee^*}\tensor{\tend{\Ee}}\Gama{\Ee}}{\jJ_{\Rep{G}}} \,\, \cong \,\, \B{\Sigma}^{\dag}\tensor{\Bb} \B{\Sigma},
\end{equation}
see the Appendix for the notation.
The  multiplication of $\Repf{G}$ is defined using the tensor product in $\Rep{G}$ and the unit is the image of $B\tensor{\Bbbk}B$ by taking the identity representation $(\Ii, \varrho^{\Ii})$.

The algebra $\Repf{G}$ is called \emph{the algebra of representative functions} on the groupoid $\gG$.  The terminology will be soon justified.
We will use the notation $\bara{\varphi\tensor{\tend{\Ee}}p}$ to denote elements in $\Repf{G}$ (here we are  taking the second description in equation \eqref{Eq:RG}). Such an element corresponds then to the $\gG$-representation $\varrho^{\Ee}$ and it is represented by the generic element $\varphi\tensor{\tend{\Ee}}p \in   \Gama{\Ee}^*\tensor{\tend{\Ee}}\Gama{\Ee}$.
\end{corodefinition}

It is convenient to give explicitly the Hopf algebroid structure of $(B,\Repf{G})$.  
The multiplication and the unit are defined as was mentioned above, using the tensor product of $\gG$-representations and the unit uses the identity representation $\varrho^{\Ii}$. 

Precisely,  given two $\gG$-representations $\varrho^{\Ee}, \varrho^{\Ff}$ and two elements of the form $\bara{\varphi \tensor{\tend{\Ee}}p}$, $\bara{\psi \tensor{\tend{\Ff}}q}$  in $\Repf{G}$ with $\varphi \in \Gama{\Ee}^*$, $p \in \Gama{\Ee}$ and $\psi \in \Gama{\Ff}^*$, $q \in \Gama{\Ff}$, we have 
$$
\bara{\varphi \tensor{\tend{\Ee}}p}\, . \, \bara{\psi \tensor{\tend{\Ff}}q} \,\, =\,\, \bara{(\varphi\tensor{B}\psi) \tensor{\mathrm{T}_{\Ee\tensor{}\Ff}}(p\tensor{B}q)},
$$ where we have  identified $\Gama{\Ee^*} \tensor{B}\Gama{\Ff^*}$ with $\Gama{(\Ee\tensor{}\Ff)^*}$ via the canonical isomorphism which sends $\varphi\tensor{B}\psi$ to the section $\gG_0 \ni x \mapsto \varphi(x)\tensor{\Bbbk}\psi(x)  \in E_x^*\tensor{\Bbbk}F_x^*$. Of course this operation is  independent form the chosen representing elements in the equivalence classes of the second description in equation \eqref{Eq:RG}.

The reminder of structure maps  are described as follows:
\begin{itemize}
\item \emph{Source and target}: $\Sf{s}: B \to \Repf{G},  \lr{ a \mapsto  \bara{1_B\tensor{\tend{\Ii}}a} }, \, \Sf{t}:B \to \Repf{G},  \lr{ a \mapsto \bara{a\tensor{\tend{\Ii}}1_B} }$.
\item \emph{Counit}: $\varepsilon: \Repf{G} \to B, \,\lr{ \varepsilon(\bara{\varphi\tensor{\tend{\Ee}}p}) : \gG_0 \to \Bbbk, \Lr{x \mapsto \varphi(x)(p(x)) } }$
\item \emph{Comultiplication}:
$\Delta: \Repf{G} \to \Repf{G}\tensor{B}\Repf{G}, \,\lr{\bara{\varphi\tensor{\tend{\Ee}}p} \mapsto \sum_i^n \bara{\varphi\tensor{\tend{\Ee}}s_i}\tensor{B} \bara{s_i^*\tensor{\tend{\Ee}}p}}$
where $\{s_i,s_i^*\}_i^n$ is the 'global' dual basis of $\Gama{\Ee}_B$ and $n$ is the rank of the bundle $\Ee$.
\item \emph{Antipode}: $\sS: \Repf{G} \to \Repf{G},\; \lr{\bara{\varphi\tensor{\tend{\Ee}}p} \mapsto \bara{\td{p}\tensor{\mathrm{T}_{\Ee^*}} {\varphi}}}$,  where we have used the isomorphism $ \td{(-)}: \Ee \, \cong \, (\Ee^*)^*$.
\end{itemize}

The following proposition justifies the terminology used for $\Repf{G}$.

\begin{proposition}\label{prop:zeta}
Let $\gG$ be a groupoid and $B=\algb{\gG}$ its base $\Bbbk$-algebra.  The following  map
\begin{equation}\label{Eq:zeta}
\xymatrix@R0pt{  \Repf{G} \ar@{->}^-{\zeta}[rr]  & & \algt{\gG} \\ \bara{\varphi\tensor{T_{\Ee}}p } \ar@{|->}[rr] & & \LR{g \longmapsto \varphi(\Sf{t}(g))\lr{
\varrho_g^{\Ee}\big(p(\Sf{s}(g))\big) }   }     }
\end{equation}
is an injective  $(B\tensor{\Bbbk}B)$-algebras map to the total algebra of $\gG$. Moreover, we have 
\begin{enumerate}[(1)]
\item $ \iota^* \circ \zeta\,\,=\,\, \varepsilon$. \label{221}
\item $ \zeta \circ \sS (\bara{\varphi\tensor{\tend{\Ee}}p }) (g) \,\,=\,\,  \zeta(\bara{\varphi\tensor{\tend{\Ee}}p })(g^{-1})$, for every   $\bara{\varphi\tensor{\tend{\Ee}}p } \in \Repf{G}$  and $g \in \gG_1$. \label{222}
\item For every pair of arrows $f, g \in \gG_1$ with $\Sf{t}(f)=\Sf{s}(g)$ and every element $F \in \Repf{G}$: $$\zeta(F)(g\circ f)\,\, =\,\, \zeta(F_1)(g)\zeta(F_2)(f), \;\text{ where } \Delta(F)=F_1\tensor{B}F_2.\,\text{(summation understood)}$$ \label{223}
\end{enumerate}
\end{proposition}
\begin{proof} 
An easy verification, using the definition of the two-sided ideal $\jJ_{\Rep{G}}$ of equation \eqref{Eq:JG} (see Appendix) and diagrams \eqref{Eq:diag}, shows that $\zeta$ is a well defined map. Since the multiplication in $\Repf{G}$ is defined via the tensor product of representations and that of $\algt{\gG}$ is defined componentwise,  we easily check that $\zeta$ is multiplicative and unital.  

We need to check that $\zeta$ is injective. To this end we define the following $(B\tensor{\Bbbk}B)$-subalgebra of the total algebra  $\algt{\gG}$. Take a $\gG$-representation $\varrho^{\Ee}$  with rank $n$, we denote by $\vV(\varrho^{\Ee})$ the  $(B\tensor{\Bbbk}B)$-sub-bimodule  of $\algt{\gG}$ generated by   the set of functions $\{a_{\scriptscriptstyle{ij}}\}_{1 \leq i,j \leq n}$ where for each $g \in \gG_1$ we have $\varrho^{\Ee}_g=(a_{\scriptscriptstyle{ij}}^g)_{1 \leq i,j \leq n}$ the $n$-square matrix representing the $\Bbbk$-linear isomorphism $\varrho^{\Ee}_g$.  
That is, an element of the form $(\lambda\tensor{\Bbbk}\gamma). a_{ij} \in \vV(\varrho^{\Ee})$ defines on $\gG_1$ the function  $g \mapsto \lambda(\Sf{t}(g)) a_{ij}^g \gamma(\Sf{s}(g))$.
Now, we define the following  $(B\tensor{\Bbbk}B)$-algebra
$ \vV(\gG)\,=\,\sum_{\varrho^{\Ee}\, \in \, \Rep{G}} \vV(\varrho^{\Ee})$ whose 
 multiplication is defined as in $\Repf{G}$ by using the tensor product of $\gG$-representations and the unit is given by the identity representation $\varrho^{\Ii}$.  
An easy computation shows now that  $\zeta(\Repf{G}) \,=\, \vV(\gG)$. 

On the other hand, for any $\gG$-representation $\varrho^{\Ee}$, we have a well defined morphism of $(B\tensor{\Bbbk}B)$-modules  
$$
\xi_{E}: \Gama{\Ee}^*\tensor{\Bbbk}\Gama{\Ee} \longrightarrow \vV(\gG),\quad \Big(\varphi \tensor{\Bbbk}p \longmapsto \sum_i\varphi_i\tensor{\Bbbk}p_j .a_{ij}\Big),
$$
where, for every $x \in \gG_0$,  we expressed  $\varphi(x) \,=\, \sum_i \varphi_i(x)s_i(x)$ in $E_{x}^*$ and $p(x)\,=\, \sum_j p_j(x)s_j(x)$ in $E_{x}$ by taking  a 'global' dual basis $\{s_i,s_i^*\}_i$ of the $B$-module $\Gama{\Ee}$. 
It turns out that the pair $(\vV(\gG), \xi_{-})$ is a dinatural transformation. Furthermore, by construction one can shows that $(\vV(\gG),\xi_{-})$ is the coend of the functor $\omega^*(-)\tensor{\Bbbk}\omega(-)$. Therefore, by the universal property we have an isomorphism\footnote{In fact an isomorphism of commutative Hopf $B$-algebroids, where $\vV(\gG)$ can be endowed with a natural structure of a   Hopf algebroid.} $\vV(\gG) \, \cong \, \Repf{G}$ via the map $\zeta$, and thus $\zeta$ is injective.

The equality of item \eqref{221}  in the last statement is clear. The second item  follows by an easy computation using the formula of equation \eqref{Eq:star}. The equality of item \eqref{223}  is obtained as follows. Take $g, f \in \gG_1$ with $\Sf{s}(g)=\Sf{t}(f)$, it suffices to check the stated equality of  elements of the form $\bara{\varphi\tensor{\tend{\Ee}}p}$ in $\Repf{G}$ given by some $\gG$-representation $\varrho^{\Ee}$ with rank $n$ and a global dual basis  $\{s_i,s_i^*\}_i^n$ of sections. So, from one hand we have  
\begin{eqnarray*}
\zeta\big(\bara{\varphi\tensor{\tend{\Ee}}p}\big) (gf)&=&  
\varphi(\Sf{t}(gf))\lr{
\varrho_{gf}^{\Ee}\big(p(\Sf{s}(gf))\big)} \\ &=&  \varphi(\Sf{t}(g))\lr{
\varrho_{gf}^{\Ee}\big(p(\Sf{s}(f))\big)} \\ & \overset{\eqref{Eq:cocycle}}{=} &   \varphi(\Sf{t}(g))\lr{
\varrho_g^{\Ee}\big(\varrho_f^{\Ee}(p(\Sf{s}(f)))\big)}
\end{eqnarray*}
On the other hand, we have 
\begin{eqnarray*}
\zeta\Big(\big(\bara{\varphi\tensor{\tend{\Ee}}p}\big)_1\Big)(g)  \, \zeta\Big(\big(\bara{\varphi\tensor{\tend{\Ee}}p}\big)_2\Big)(f) &=&  \sum_i  \zeta\big(\bara{\varphi\tensor{\tend{\Ee}}s_i} \big)(g) \,  \zeta\big(\bara{s_i^*\tensor{\tend{\Ee}}p} \big)(f) \\ & = &  \sum_i \varphi(\Sf{t}(g))\lr{
\varrho_g^{\Ee}\big(s_i(\Sf{s}(g))\big)}
s_i^*(\Sf{t}(f))\lr{
\varrho_f^{\Ee}\big(p(\Sf{s}(f))\big)} \\ &=& 
\varphi(\Sf{t}(g))\lr{ \sum_i
\varrho_g^{\Ee}\big(s_i(\Sf{s}(g))\big) s_i^*(\Sf{t}(f))\lr{
\varrho_f^{\Ee}\big(p(\Sf{s}(f))\big)}} \\ &=& 
\varphi(\Sf{t}(g))\lr{ 
\varrho_g^{\Ee}\lr{ \sum_i s_i(\Sf{s}(g))s_i^*(\Sf{t}(f))\lr{
\varrho_f^{\Ee}\big(p(\Sf{s}(f))\big)}}} \\ &=& 
\varphi(\Sf{t}(g))\lr{ 
\varrho_g^{\Ee}\big(
\varrho_f^{\Ee}\big(p(\Sf{s}(f))\big)\big)}
\end{eqnarray*}
Therefore, $\zeta\big(\bara{\varphi\tensor{\tend{\Ee}}p}\big) (gf) = \zeta\Big(\big(\bara{\varphi\tensor{\tend{\Ee}}p}\big)_1\Big)(g)  \, \zeta\Big(\big(\bara{\varphi\tensor{\tend{\Ee}}p}\big)_2\Big)(f)$ which finishes the proof.
\end{proof}

\begin{example}\label{exmp:UV}
If we start with a groupoid whose source is  equal to its target, in other words,  a  disjoint union of groups (or group bundle) $\cup_{x \in X}G_x$ parametrized by a non empty set $X$. Then the associated algebra of representative functions is in this case a commutative Hopf algebra over the base algebra $B={\rm M}_{\Bbbk}(X)$ (i.e. a Hopf algebroid with same source and target).  In this case if $X$ is reduced to a point, which means that we are given a single group $G$. Then $B=\Bbbk$ and  by Example \ref{exam:action}, $\mathscr{R}_{\Bbbk}(G)$ is exactly the commutative Hopf  $\Bbbk$-algebra of representative functions on $G$. That is, it coincides with the \emph{finite dual} $(\Bbbk G)^o$ of the group algebra $\Bbbk G$, see \cite{Abe:book, Hochschild:book}.

A trivial case of the disjoint union of groups  $\cup_{x \in X}G_x$ is when each group of the $G_x$'s has only one element (the neutral element). Thus, for any non empty set $X$ one can consider the groupoid $\Uu(X)$ known as \emph{the unit groupoid} of $X$\footnote{The terminology is that of \cite{Cartier:2008}.}, where  $\Uu(X)_1=\Uu(X)_0=X$ and $\Sf{t}=\Sf{s}=\iota=id_X$. A finite dimensional $\Uu(X)$-representation is nothing but a set of the form $X \times V$ where $V$ is a finite dimensional $\Bbbk$-vector space.  Under this description, the Hopf $B$-algebra of representative functions is given by following quotient of $B$-module
$$
\mathscr{R}_{\Bbbk}(\Uu(X))\,\, =\,\, \frac{ \lr{\bigoplus_{n \in \mathbb{N}}B^n  \tensor{M_n(\Bbbk)}B^n}}{ \Big\langle \underset{}{}  u  \tensor{M_n(\Bbbk)} (\lambda_{\scriptscriptstyle{ i j}}) v - u (\lambda_{\scriptscriptstyle{ ij}})\tensor{M_m(\Bbbk)} v \Big \rangle_{u \in B^n,\, v\in B^m,\, (\lambda_{\scriptscriptstyle{ ij}})\, \in M_{n,m}(\Bbbk)}}.
$$ 

Another less trivial groupoid, is the \emph{groupoid of pairs}  $\Vv(X)$ (or \emph{fine groupoid} in the terminology of \cite{Brown:1987} which is a particular case of \emph{principal groupoids} \cite{Renault:1980}). Here $\Vv(X)_1=X \times X$, $\Vv(X)_0=X$ with   $\Sf{s}=pr_1$, $\Sf{t}=pr_2$ (the first and second projections)   and $\iota=\delta$ (the diagonal map). The composition is understood.  In this case a finite dimensional $\Vv(X)$-representation is a bundle $\cup_{x \in X}E_x$ with a 'type fiber' an  $n$-dimensional  $\Bbbk$-vector space $V$, together with a family of invertible $n$-square matrices $\{(a_{\scriptscriptstyle{ij}}^{(x,y)})_{\scriptscriptstyle{i,j}}\}_{x,y \in \, X}$ in $M_n(\Bbbk)$ satisfying
$$
\lr{a_{\scriptscriptstyle{ij}}^{(x,y)}}_{\scriptscriptstyle{i,j}} \, \lr{a_{\scriptscriptstyle{ij}}^{(y,z)}}_{\scriptscriptstyle{i,j}}\,=\, \lr{a_{\scriptscriptstyle{ij}}^{(y,z)}}_{\scriptscriptstyle{i,j}},\quad \lr{a_{\scriptscriptstyle{ij}}^{(x,x)}}_{\scriptscriptstyle{i,j}}\,=\, I_n, \,\, \text{ for every }\, x, y, z \in X. 
$$
The description of the commutative Hopf algebroid $\mathscr{R}_{\Bbbk}(\Vv(X))$  is approximately similar to that of $\mathscr{R}_{\Bbbk}(\Uu(X))$.  Thus, instate of taking the $n$-square matrix algebra in the above nominator of $\mathscr{R}_{\Bbbk}(\Uu(X))$, we take  the following $\Bbbk$-subalgebra of $M_n(B)$ consisting of matrices which are 'locally' conjugated by the $(a_{\scriptscriptstyle{i,j}})_{\scriptscriptstyle{i,j}}$'s. That is, matrices $(\lambda_{\scriptscriptstyle{ij}})_{\scriptscriptstyle{i,j}} \in M_n(B)$  such that
$$ \lr{\lambda_{\scriptscriptstyle{ij}}^y}_{\scriptscriptstyle{i,j}} \,  \lr{a_{\scriptscriptstyle{ij}}^{(x,y)}}_{\scriptscriptstyle{i,j}}\,\, =\,\, \lr{a_{\scriptscriptstyle{ij}}^{(x,y)}}_{\scriptscriptstyle{i,j}} \,  \lr{\lambda_{\scriptscriptstyle{ij}}^x}_{\scriptscriptstyle{i,j}},\,\, \text{ for every } x, y \in X.$$

It is noteworthy to mention, that  a  more precise and complete  description of both Hopf  algebroids,  by means of generators and relations,   is far from being obvious even under strong hypothesis of finiteness on the set $X$.  
\end{example}

\begin{example}\label{exam:FunctTransGrp}
Let  $H$ be a group with identity element $e$, and $X$ any set. Consider the following groupoid
$$
\gG:\;\xymatrix@C=50pt{X \times H \times X \ar@<1.2ex>@{->}|-{\scriptstyle{pr_1}}[r] \ar@<-1.2ex>@{->}|-{\scriptstyle{pr_3}}[r] & \ar@{->}|-{ \scriptstyle{\iota} }[l] X, }
$$
with $\Sf{s}=pr_1$ and  $\Sf{t}=pr_3$, the first and third projections,  and where $\iota(x) =(x,e,x)$, for every $x \in X$. The composition and the inverse maps are given by 
$$
(x,g,y)\, .\, (y, h, z) \,\, =\,\, (x, gh, z), \qquad (x,g,y)^{-1}\,=\, (y,g^{-1},x).
$$

Let $\rR_{\Bbbk}(H)$ be the Hopf $\Bbbk$-algebra of representative functions on $H$ which we  consider as a Hopf algebroid with source equal target. Consider then  its  extended Hopf algebroid $\big({\rm M}_{\Bbbk}(X),  {\rm M}_{\Bbbk}(X)\tensor{\Bbbk} \rR_{\Bbbk}(H) \tensor{\Bbbk} {\rm M}_{\Bbbk}(X)\big)$.  One can easily check then that the image of $\Repf{G}$ in the total algebra  
${\rm M}_{\Bbbk}(X\times H \times X)$  by the map $\zeta$ of Proposition \ref{prop:zeta} coincides with  the image of the canonical map $${\rm M}_{\Bbbk}(X)\tensor{\Bbbk} \rR_{\Bbbk}(H) \tensor{\Bbbk} {\rm M}_{\Bbbk}(X) \hookrightarrow  {\rm M}_{\Bbbk}(X\times H \times X).$$
Therefore, there is an isomorphism of ${\rm M}_{\Bbbk}(X)$-bimodules 
$\Repf{G} \cong {\rm M}_{\Bbbk}(X)\tensor{\Bbbk} \rR_{\Bbbk}(H) \tensor{\Bbbk} {\rm M}_{\Bbbk}(X)$ 
which by the universal property of $\Repf{G}$ is an isomorphism of Hopf ${\rm M}_{\Bbbk}(X)$-algebroids. Furthermore, the canonical  morphism of  Hopf algebroids 
$$
(\Bbbk, \rR_{\Bbbk}(H)) \longrightarrow \big({\rm M}_{\Bbbk}(X),  {\rm M}_{\Bbbk}(X)\tensor{\Bbbk} \rR_{\Bbbk}(H) \tensor{\Bbbk} {\rm M}_{\Bbbk}(X)\big) 
$$
coincides, up to this isomorphism, with the morphism $\B{\Rr}(pr_2): (\Bbbk, \rR_{\Bbbk}(H)) \to \big({\rm M}_{\Bbbk}(X),  \Repf{G}\big) $  of Hopf algebroids  which corresponds to the obvious morphism of groupoids $pr_2: \gG \to H$.
\end{example}

Now we come back to the general situation. So let $\gG$ be a groupoid with base algebra $B=\algb{\gG}$. 
By applying Deligne's Theorem \cite[Théorème 5.2]{Bruguieres:1994} together with \cite[Théorème 7.1]{Bruguieres:1994}  to Corollary \ref{coro:1}, we know that $(B,\Repf{G})$ is a \emph{transitive commutative  Hopf algebroid}, and so it is \emph{geometrically transitive} by \cite[Théorème 8.2]{Bruguieres:1994}, see also Lemma \ref{lema:LK} in the Appendix.  More consequences of this property are stated in Example \ref{exam:Gg}(2).

\begin{proposition}\label{prop:GT}
Let $\gG$ be a groupoid with $B$ its base $\Bbbk$-algebra.  Then the algebra  of representative functions $\Repf{G}$ is a geometrically transitive Hopf  $B$-algebroid. Equivalently, $\Repf{G}$ is a projective and faithfully flat $(B\tensor{\Bbbk}B)$-module.
\end{proposition}
\begin{proof}
This is a direct consequence of \cite[Théorèmes 5.2 et  7.1]{Bruguieres:1994} (see also Lemma \ref{lema:LK}).
\end{proof}

\begin{remark}\label{rem:RFTops}
Let us consider now a  topological groupoid $\gG$ with compact Hausdorff base space $\gG_0$.  Take its category of representations $\tRep{G}$ with the "fiber" functor $\omega^{top}: \tRep{G} \to \cproj{\gG}$ as in Remark \ref{rem:NotTanaka} ($\Bbbk$ is the field of  real or complex numbers).  The construction of the  $(\calgb{\gG}\tensor{\Bbbk}\calgb{\gG})$-algebra of \emph{ continuous representative functions} is also possible in this context. In fact, associated to the pair $(\tRep{G},\omega^{top})$ we can construct, in similar way as in   equation \eqref{Eq:Lomega} (see Appendix below), a commutative Hopf $\calgb{\gG}$-algebroid $\tRepf{G}$ which of course in this case is no longer geometrically transitive. The image of $\tRepf{G}$ by the  algebra map $\zeta$  of Proposition \ref{prop:zeta}, lands now in the total algebra $\calgt{\gG}$ of all continuous functions from $\gG_1$ to $\Bbbk$. The fact that $\tRep{G}$ is no longer abelian prevents us to apply Deligne-Bruguières's machinery to well understand the geometric nature of the commutative Hopf algbebroid $\tRepf{G}$. This is why perhaps this context suggests then to study  pseudo-Tannakian $\Bbbk$-linear categories (see Remark \ref{rem:NotTanaka}) in relation with commutative $C^*$-algebras. 
\end{remark}

\section{Duality between discrete groupoids and Hopf algebroids.}\label{sec:3}
This section contains our main result, namely, Theorem \ref{thm:main}. We show that there is a duality between  the category of discrete groupoids and the category of geometrically transitive Hopf algebroids.  By a duality we mean here a kind of an adjunction between contravariant functors. Specifically, let $\mathscr{A}$ and $\mathscr{B}$ two additive categories. We say that  there is a duality between $\mathscr{A}$  and $\mathscr{B}$,  if there exists a pair $\xymatrix{F:  \mathscr{A} \ar@<-1.2pt>[r] & \ar@<-1.2pt>[l] \mathscr{B} : G}$ of contravariant additive functors together with natural transformations $\theta: id_{\mathscr{A}} \to GF$, $\eta:id_{\mathscr{B}} \to FG$ such that $F\theta \circ \eta_F=F$ and $G\eta \circ \theta_G=G$.

\subsection{The representative functions functor $\mathscr{R}_{\Bbbk}: \Grpd^{\scriptscriptstyle{op}} \to \CHAlgd_{\Bbbk}$.}\label{ssec:R}
We denote by $\Grpd$ the category of discrete groupoids and by $\CHAlgd_{\Bbbk}$ the category of commutative Hopf algebroids with ground field $\Bbbk$, that is, all the  involved commutative rings are $\Bbbk$-algebras.  

Let $\phi: \gG \to \hH$ be a morphism of groupoids. To distinguish between the fiber functors of $\gG$ and $\hH$,  we use the following notations  $\omega^{\gG}: \Rep{G} \to \proj{\gG}$ and $\omega^{\hH}: \Rep{H} \to \proj{\hH}$. We consider $\algb{\gG}$ as an $\algb{\hH}$-module via the extension ${\rm M}_{\Bbbk}(\phi_{\scriptscriptstyle{0}}): \algb{\hH} \to \algb{\gG}$.  Our aim here is to show that representative functions is a functorial construction.

\begin{lemma}\label{lema:cuadro}
Let $\phi: \gG \to \hH$ be a morphism of  groupoids and consider  the restriction functor
$\B{\Rr}(\phi)$ as defined in Lemma \ref {lema:indfunct}. Then the following diagram of functors
$$
\xymatrix@C=60pt{ \Rep{H} \ar@{->}^-{\B{\Rr}(\phi)}[rr] \ar@{->}_-{\omega^{\hH}}[d] & & \Rep{G} \ar@{->}^-{\omega^{\gG}}[d] \\ \proj{\hH} \ar@{->}_-{{\rm M}_{\Bbbk}(\phi_0)^*}[rr] & & \proj{\gG}}
$$ is strictly commutative.
\end{lemma}
\begin{proof}
Let $\varrho^{\Ff}$ be an $\hH$-representation. We know that the underlying bundle of $\B{\Rr}(\phi)(\varrho^{\Ff})$ is the induced  (pull-back) bundle $\phi_{\scriptscriptstyle{0}}^*(\Ff)$. Therefore, its 'global' sections are given by the tensor product module $\B{\Gamma}(\phi_{\scriptscriptstyle{0}}^*(\Ff))\,=\, \B{\Gamma}(\Ff)\tensor{\algb{\hH}}\algb{\gG}$. The same happens to morphisms between $\hH$-representations. That is, we have $\omega^{\hH} \circ \B{\Rr}(\phi)\,=\,  (-\tensor{\algb{\hH}}\algb{\gG}) \circ \omega^{\gG}$ and the state diagram is commutative.
\end{proof}

\begin{proposition}\label{prop:R}
The assignment $\mathscr{R}_{\Bbbk}:  \Grpd \to \CHAlgd_{\Bbbk}$ which sends any discrete groupoid $\gG$ to its Hopf algebroid $(\algb{\gG},\Repf{G})$ of representative functions, is a well defined contravariant functor with image in the full subcategory of geometrically transitive Hopf algebroids. Furthermore, for any  morphism of groupoids $\phi:\gG\to \hH$, we have the following commutative diagram of algebras
$$
\xymatrix@C=40pt{ \Repf{H} \ar@{->}^-{\mathscr{R}_{\Bbbk}(\phi)}[rr] \ar@{->}_-{\zeta_{\hH}}[d]  & & \Repf{G} \ar@{->}^-{\zeta_{\gG}}[d] \\ \algt{\hH} \ar@{->}_-{{\rm M}_{\Bbbk}(\phi_1)}[rr] & & \algt{\gG}} 
$$ where $\zeta$ is the algebra map of Proposition \ref{prop:zeta}.
\end{proposition}
\begin{proof} The first part of this proof is in fact a sketch of the proof of the general statement given in  Lemma \ref{lema:LK}, see Appendix.

For simplicity we denote here the functor $\B{\Rr}(\phi)$ by $(-){\phi}$. So let $\varrho^{\Ff}$ be any $\hH$-representation, using Lemma \ref{lema:cuadro}, we can show that there is a morphism of  $\algb{\hH}$-bimodules:
$$
\xymatrix{ \omega^{\hH}(\varrho^{\Ff})^*\tensor{\Bbbk}\omega^{\hH}(\varrho^{\Ff}) \ar@{-->}_-{\theta_{\Ff}}[rrd]  \ar@{->}_-{}[rr] & & 
\omega^{\gG}(\varrho^{\Ff_{\phi}})^*\tensor{\Bbbk}\omega^{\gG}(\varrho^{\Ff_{\phi}}) \ar@{->}_-{}[d] \\ & & \Repf{G} }
$$ where the horizontal arrow is clear and the vertical arrow is one of the canonical morphisms of $\algb{\gG}$-bimodules defining the universal object $\Repf{G}$. 

It turns out that the pair $(\Repf{G},\theta_{\Ff})_{\varrho^{\Ff} \in \, \Rep{H}}$ is a dinatural transformation for the bi-functor $\omega^{\hH}(-)^*\tensor{\Bbbk}\omega^{\hH}(-)$ in $\algb{\hH}$-bimodules. Therefore, by the universal property of $\Repf{H}$, there is a morphism of $\algb{\hH}$-bimodules $\mathscr{R}_{\Bbbk}(\phi): \Repf{H} \to \Repf{G}$ which in fact is a morphism of Hopf algebroids. The remainder axioms on $\mathscr{R}_{\Bbbk}$ to be a functor, are easily checked. The commutativity of the stated diagram is directly obtained from the basic properties of the restriction  functor $\B{\Rr}$ given in Lemma  \ref{lema:indfunct}.
\end{proof}

Before given examples concerning Hopf algebroids of representative functions, let us first recall the notion of \emph{induced groupoid} which we will use in the sequel.

\begin{example}\label{exam:u}
Let $\hH$ be a groupoid and $u: P \to \hH_0$ any map. Consider the following set
$$ P_1:=P\,{}_{u}\times_{\Sf{s}}\,\hH_1\, {}_{\Sf{t}}\times_{u} \,P\,\,=\LR{(p,g,q) \in P\times \hH_1 \times P|\,\, u(p)=\Sf{s}(g), \, u(q)=\Sf{t}(g)}.$$
Clearly $(P_1,P)$ has a structure of groupoid\footnote{If $\hH$ is a groupoid with only one object, that is, a group, and $P$ is any set then one recovers the groupoid defined in Example \ref{exam:FunctTransGrp}.}and $(pr_2,u): (P_1,P) \to (\hH_1,\hH_0)$ is a morphism of groupoids, where $pr_2: P_1 \to \hH_1$ is the second projection. This is the \emph{induced groupoid} by $u$ and denoted by $\hH_u$. For instance, if we  fix a  point $x \in \hH_0$, we  can then  take $P=\{x\}$ and show that the groupoid $(\{x\}_1,\{x\})$ is in fact a group which coincides with the isotropy group of $\hH$ at $x$.

Furthermore, any morphism $\phi: \gG \to \hH$ of groupoids can be lifted to the induced groupoid by $\phi_0$. That is, we have a morphism of groupoids $\psi: \gG \to \gG_{\phi_{\scriptscriptstyle{0}}}$, where $\psi_0=\phi_0$ and $\psi_1: \gG_1 \to \gG_0\,{}_{\phi_0}\times_{\Sf{s}}\, \hH_1\,{}_{\Sf{t}}\times_{\phi_0}\, \gG_0$ sends $g \mapsto (\Sf{s}(g), \phi_1(g), \Sf{t}(g))$.
\end{example}

\begin{examples}\label{exam:Gg} Since groups are  in a canonical  way related with groupoids, it is natural to try to see, using the functor $\mathscr{R}_{\Bbbk}$,   how this relation behave at the level of  Hopf algebras and Hopf algebroids (here Hopf algebra are consider as Hopf algebroid with source equal target).  
\begin{enumerate}[(1)]
\item Let $\gG$ be a  groupoid and consider the (full) sub-groupoid $\gG^{i} \hookrightarrow \gG$ consisting of arrows $\gG^i{}_1=\{ g \in \gG_1|\, \Sf{s}(g)=\Sf{t}(g)\}$ and objects $\gG^i{}_0=\gG_0$ (i.e., the \emph{isotropy groupoid of $\gG$}). Then we have a surjective map $\Repf{G} \twoheadrightarrow \mathscr{R}_{\Bbbk}(\gG^i)$ of Hopf $\algb{\gG}$-algebroids, where $\Repf{\gG^i}$ is  the  quotient Hopf $\algb{\gG}$-algebra of $\Repf{G}$ by the ideal generated by the set $\{b \in \algb{\gG}|\, \Sf{s}(b)-\Sf{t}(b)\}$. 
\item Let $\gG$ be a  groupoid and fix a point $x \in \gG_0$. Consider the isotropy group at $x$, that is, the group $G_x=\{g \in \gG_1|\, \Sf{s}(g)=\Sf{t}(g)=x\}$. This point gives a rise to the $\Bbbk$-algebra map $ev_x: \algb{\gG} \to \Bbbk$ sending $a \mapsto a(x)$. We denote by $\Bbbk_x$ the $\algb{\gG}$-algebra $\Bbbk$ via the extension $ev_x$. In this way the monorphism of groupoids $\Sf{x}=(i_x,x): (G_x,\{x\}) \hookrightarrow \gG$, where we denote $x:\{x \} \hookrightarrow \gG_0$, leads to a morphism of Hopf algebroid  
\begin{equation}\label{Eq:isotropy}
(ev_x,\mathscr{R}_{\Bbbk}(\Sf{x})): (\algb{\gG},\Repf{G}) \longrightarrow (\Bbbk_x,\mathscr{R}_{\Bbbk}(G_x)).
\end{equation} 
This Hopf $\Bbbk$-algebra is called \emph{the isotropy Hopf algebra} at the point $ev_x$ (which in fact corresponds to the isotropy group at $x$).

On the other hand, it is easily checked that  the base extension Hopf algebroid $(\Bbbk_x, \Bbbk_x\tensor{\algb{\gG}}\Repf{G}\tensor{\algb{\gG}}\Bbbk_x)$ coincides with the Hopf algebra $(\Bbbk_x, \mathscr{R}_{\Bbbk}(\gG_{x}))$ of the induced groupoid $\gG_x$ by $x$.
Therefore,  it  is isomorphic to the isotropy Hopf $\Bbbk$-algebra $\mathscr{R}_{\Bbbk}(G_{x})$, since $\gG_x$ and  $G_x$ are isomorphic as groups, see Example \ref{exam:u}.

The geometrically transitive property of the Hopf algebroid $(\algb{\gG},\Repf{G})$ could be then interpreted by saying that, for any point $x \in \gG_0$, the extended map  of Hopf $\Bbbk$-algebras  \begin{equation}\label{Eq:extension}
\Bbbk_x\tensor{\algb{\gG}}\Repf{G}\tensor{\algb{\gG}}\Bbbk_x \longrightarrow \mathscr{R}_{\Bbbk}(G_x)  
\end{equation}
is an isomorphism and that the category of $\Repf{G}$-comodules is  equivalent (as symmetric monoidal category) to the category of $\mathscr{R}_{\Bbbk}(G_x)$-comodules, as was expound in  \cite{Deligne:1990}.

\item Let $X$ be an non empty set and consider its  unit groupoid $\Uu(X)$ together with its  groupoid of pairs  $\Vv(X)$, see Example \ref{exmp:UV}. The diagonal map clearly gives a monomorphism of groupoids $\Uu(X) \hookrightarrow \Vv(X)$. Then we have a morphism of Hopf ${\rm M}_{\Bbbk}(X)$-algebroids $\mathscr{R}_{\Bbbk}(\Vv(X)) \to   \mathscr{R}_{\Bbbk}(\Uu(X))$ which in fact can be constructed elementarywise using the description offered in  Example \ref{exmp:UV}. Although, $\mathscr{R}_{\Bbbk}(\Vv(X))$  and $\mathscr{R}_{\Bbbk}(\Uu(X))$ are not isomorphic  part of their structures behave similar in the sense that they have the same class of isotropy Hopf algebras.
\item Consider an  action groupoid $\gG$ with group $G$ acting on non empty set $X$.  As we have seen in  Example \ref{exam:actionI}, the projection $G\times X \to G$ is a morphism of groupoids which then leads to a morphism of Hopf algebroids $(\Bbbk, \mathscr{R}_{\Bbbk}(G)) \longrightarrow ({\rm M}_{\Bbbk}(X), \mathscr{R}_{\Bbbk}(G\times X))$. At this level of generality, one can not expect to   have an isomorphism $\mathscr{R}_{\Bbbk}(G\times X)\cong {\rm M}_{\Bbbk}(X)\tensor{\Bbbk}\mathscr{R}_{\Bbbk}(G)$ of Hopf ${\rm M}_{\Bbbk}(X)$-algebbroid, since in general there is no way to convert ${\rm M}_{\Bbbk}(X)$ into  $\mathscr{R}_{\Bbbk}(G)$-comodule algebra. 
This happens perhaps  only when   one is restricted to a special class of representative functions  and assuming a more rich structure on both  $G$ and $X$, see the forthcoming examples.
\end{enumerate}
\end{examples}

\begin{example}\label{exam:algGroups}
Assume  that $\gG$ is an algebraic affine action $\Bbbk$-groupoid with $\Bbbk$ an infinite field. Precisely,  consider $\gG_1=G \times X$, where $G$ is an algebraic affine $\Bbbk$-group acting on an affine algebraic $\Bbbk$-set  $X=\gG_0$  and assume that the action is a morphism of algebraic $\Bbbk$-sets.   Then the algebra $\pP(G\times X)$ of polynomial representative functions on $G \times X$ can be endowed within a Hopf $\pP(X)$-algebroid structure which   splits into a tensor product  $\pP(X)\tensor{\Bbbk}\pP(G)$  (see \cite{Hochschild:book} for basic definitions and notations). This is of course  an example of a topological action groupoid $\gG$ where we have a chain of homomorphisms of Hopf algebroids $$\big(\pP(X), \pP(G\times X)\big) \subseteq \big(\calgb{\gG},\tRepf{G}\big) \subseteq \big(\algb{\gG},\Repf{G}\big).$$

In this situation there is no a clearer  way to assert that $\pP(G\times X)$ is geometrically transitive Hopf  algebroid. However, if we assume that the $G$-action is free and transitive, then $X \to X/G$ is a principal $G$-bundle and we have an isomorphism $\pP(X)\tensor{\Bbbk}\pP(X) \to \pP(G\times X)$ of Hopf algebroids, From which we deduce that  $\pP(G\times X)$ is geometrically transitive, since so is $\pP(X)\tensor{\Bbbk}\pP(X)$. 

On the other hand, the geometric transitive property of the Hopf algebroid $\pP(G\times X)$ could be related   as  follows to the Equivariant Serre Problem \cite{Masuda/Petrie:1995}.  Assume that $X$ is the affine $n$-space and that $\pP(G\times X)$ is   geometrically transitive. Then any  object in the category $\frcomod{\pP(G\times X)}$  is finitely generated and projective $\pP(X)$-module and thus free by Quillen-Suslin's theorem. In particular, this means  that any algebraic $G$-equivariant bundle is trivial. 

There is also another situation where  the  Equivariant Serre Problem for general affine algebraic $\Bbbk$-set $X$, have a positive answer. Specifically, since we know that there is a Hopf algebroid morphism $(\eta_0,\eta_1): (\Bbbk, \pP(G)) \to (\pP(X), \pP(G\times X))$, we can associate to it the induction functor $\eta^*: \rcomod{\pP(G)} \to \rcomod{\pP(G \times X)}$ which have the ad-induction functor $\eta_{*}$ as right adjoint. So if we assume that the counit of this adjunction is a natural isomorphism, then any object $M$ in the category  $\frcomod{\pP(G\times X)}$ is isomorphic to $ \eta_{*}(M)\tensor{\Bbbk}\pP(X)\cong  M$, and thus it is a free $\pP(X)$-module, which in particular means  that any algebraic $G$-equivariant bundle is trivial.
\end{example}

\begin{example}\label{exam:TransGpd}
Let $\gG$ be a transitive groupoid, that is, the map $(\Sf{s},\Sf{t}): \gG_1 \to \gG_0 \times \gG_0$ is surjective. Fix an object $x \in \gG_0$ and choose a family of arrows $\{\tau_y\}_{y \in \gG_0}\, \subseteq \, \gG_1$ where each $\tau_y \in \Sf{t}^{-1}(\{x\})$ and $\Sf{s}(\tau_y)=y$, for $y \neq x$ and $\tau_x=\iota(x)$, for $y=x$ . It is well known (see for instance \cite{Brown:1987}), that there is a (non canonical) isomorphism of groupoids
$$
\xymatrix{\phi^x: (\gG_1, \gG_0) \ar@{->}^-{\cong}[r] &    \big(\gG_0 \times G_x \times \gG_0 , \gG_0 \big) &  \lr{ g \longmapsto \big(\Sf{s}(g), \tau_{\Sf{t}(g)} g\tau_{\Sf{s}(g)}{}^{-1}, \Sf{t}(g) \big), id_{\gG_0} }}
$$
where the right-hand side groupoid is the one defined as in Example \ref{exam:FunctTransGrp}. By applying the functor  $\mathscr{R}_{\Bbbk}$ and the isomorphism established in Example \ref{exam:FunctTransGrp},   we obtain then a chain of isomorphism of Hopf algebroids
$$
(\algb{\gG}, \Repf{G}) \, \cong \,  \big(\algb{\gG},\mathscr{R}_{\Bbbk}(\gG_0 \times G_x \times \gG_0 ) \big) \, \cong \,  \big(\algb{\gG}, \algb{\gG}\tensor{\Bbbk}\mathscr{R}_{\Bbbk}(G_x )\tensor{\Bbbk}\algb{\gG} \big)
$$
whose composition  turns to be the extended Hopf $\algb{\gG}$-algebroids morphism of the Hopf algebra isomorphism given in \eqref{Eq:extension}.

On the other hand, as was shown in Example  \ref{exam:Gg}(2), for a general groupoid $\gG$ the geometrically transitive property of its Hopf algebroid $(\algb{\gG}, \Repf{G}) $ says that,  up to isomorphisms of Hopf algebras, we only have one type of  isotropy Hopf algebras. However, there is no way to give an explicit description of such isomorphisms.  In difference when $\gG$ is  transitive,  these isomorphisms are precisely given by conjugating the isotropy groups of $\gG$, before applying the functor $\mathscr{R}_{\Bbbk}$.
\end{example}

\subsection{The character functor $\mathscr{X}_{\Bbbk}: \CHAlgd_{\Bbbk}^{\scriptscriptstyle{op}} \to \Grpd$.}\label{ssec:X}
In this section we recall the definition of the character functor.
All Hopf algebroids are considered over the ground field $\Bbbk$. We  are implicitly assuming that the base $\Bbbk$-algebra $R$ of any Hopf algebroid have the property that $\Alg{R}{\Bbbk} \neq \emptyset$.

Let $(R,\Hh)$ be  a Hopf algebroid. As in the case of Hopf algebra over fields, it is well known that  we can define the \emph{characters groupoid} of $(R,\Hh)$  as the groupoid 
$$
\chara{(R,\Hh)}:\, \xymatrix@C=50pt{\Alg{\Hh}{\Bbbk} \ar@<1.2ex>@{->}|-{\scriptstyle{\Sf{s}^*}}[r] \ar@<-1.2ex>@{->}|-{\scriptstyle{\Sf{t}^*}}[r] & \ar@{->}|-{ \scriptstyle{\varepsilon^*} }[l] \Alg{R}{\Bbbk}, }
$$ 
by dualizing the source,  target and the counit of $(R,\Hh)$. The rest of the axioms defining the underlying category of this groupoid are easily verified, whence we have identified $\Alg{(\Hh\tensor{R}\Hh)}{\Bbbk}$ with $\Alg{\Hh}{\Bbbk}{ }_{\Sf{s}^*}\times _{\Sf{t}^*} \Alg{\Hh}{\Bbbk}$ using $\Hh$ as an $R$-bimodule with $\Sf{t}$ acting on the left and $\Sf{s}$ on the right. The inverse map is given by $\sS^*: \Alg{\Hh}{\Bbbk} \to \Alg{\Hh}{\Bbbk}$, where $\sS: \Hh \to \Hh$ is the antipode of $(R,\Hh)$. 

Now, let $\alpha=(\alpha_0,\alpha_1): (R,\Hh) \to (S,\Kk)$ be a morphism of Hopf algebroids. Then clearly we obtain a morphism between character groupoids:
$$
\xymatrix@C=50pt{ \chara{(S,\Kk)}: \ar@{-->}_-{\chara{(\alpha)}}[d] & \Alg{\Kk}{\Bbbk} \ar@{->}_-{\alpha_1^*}[d]  \ar@<1.2ex>@{->}|-{\scriptstyle{\Sf{s}^*}}[r] \ar@<-1.2ex>@{->}|-{\scriptstyle{\Sf{t}^*}}[r] & \ar@{->}|-{ \scriptstyle{\varepsilon^*} }[l] \Alg{S}{\Bbbk} \ar@{->}^-{\alpha_0^*}[d]  \\ \chara{(R,\Hh)}: &
\Alg{\Hh}{\Bbbk} \ar@<1.2ex>@{->}|-{\scriptstyle{\Sf{s}^*}}[r] \ar@<-1.2ex>@{->}|-{\scriptstyle{\Sf{t}^*}}[r] & \ar@{->}|-{ \scriptstyle{\varepsilon^*} }[l] \Alg{R}{\Bbbk}. }
$$ 

\begin{lemma}\label{lema:X}
The  characters groupoid   $$\chara: \CHAlgd_{\Bbbk} \longrightarrow \Grpd$$  establishes a well defined contravariant functor. 
\end{lemma}\begin{proof} Straightforward. \end{proof}

\subsection{The unit and counit of the duality.}\label{ssec:UC}
The restriction of the characters groupoid functor to the full subcategory of geometrically transitive Hopf algebroids $\GTCHAlgd_{\Bbbk}$ will be also denoted by $\chara$. 
Let $\gG$ be an object in $\Grpd$ and consider the Hopf algebroid $(\algb{\gG}, \Repf{G})$ of representative functions on $\gG$ which by Proposition \ref{prop:GT} we know that it is  an object in $\GTCHAlgd_{\Bbbk}$. 

Let us define the following two maps $ \B{\Theta}_{\gG_0}=ev: \gG_0 \to \Alg{\algb{\gG}}{\Bbbk}$ by evaluating on each point of $\gG_0$ and 
$\B{\Theta}_{\gG_1}: \gG_1 \to \Alg{\Repf{G}}{\Bbbk}$ by applying the map $\zeta$ of Proposition \ref{prop:zeta}. Explicitly, using the notation of  Proposition \ref{prop:zeta}, $\B{\Theta}_{\gG_1}$ is defined by sending 
$$ \gG_1 \ni g \longmapsto \Lr{\bara{\varphi\tensor{T_{\Ee}}p } \longmapsto \varphi(\Sf{t}(g))\lr{
\varrho_g^{\Ee}\big(p(\Sf{s}(g))\big)}} \, \in \Alg{\Repf{G}}{\Bbbk}.$$

\begin{lemma}\label{lema:unit}
Keep the above notations. We have
\begin{enumerate}[(i)]
\item For each groupoid $\gG$, the pair $(\B{\Theta}_{\gG_0}, \B{\Theta}_{\gG_1}) : \gG \to \mathscr{X}_{\Bbbk} \circ \mathscr{R}_{\Bbbk}(\gG)$  is a morphism of groupoids. 
\item $\B{\Theta}_{-}: id_{\Grpd} \to \mathscr{X}_{\Bbbk} \circ \mathscr{R}_{\Bbbk}$ is a  natural transformation. 
\end{enumerate}
\end{lemma}
\begin{proof}
$(i)$.  The compatibility of $\B{\Theta}_{\gG}$ with  both sources and targets of the two groupoids  is easily checked.  To verify the compatibility of $\B{\Theta}_{\gG}$ with the identity and the inverse maps, we use the first two items stated in 
Proposition \ref{prop:zeta}.
Let us then check that $\B{\Theta}_{\gG}$ is  compatible with the compositions.
Take $g, f \in \gG_1$ with $\Sf{s}(g)=\Sf{t}(f)$, and fix an element of the form $\bara{\varphi\tensor{\tend{\Ee}}p}$ in $\Repf{G}$ given by the $\gG$-representation $\varrho^{\Ee}$ with rank $n$ and a global dual basis  $\{s_i,s_i^*\}_i^n$ of the  sections of its underlying bundle. 
Since we know that $$\B{\Theta}_{\gG_1}(h)\big(\bara{\varphi\tensor{\tend{\Ee}}p}\big)\,=\, \zeta\big(\bara{\varphi\tensor{\tend{\Ee}}p}\big)(h), \quad \forall h \, \in \gG_1,$$ we have, by applying the equality of item (3) in Proposition \ref{prop:zeta}, that 
\begin{eqnarray*}
\B{\Theta}_{\gG_1}(gf)\big(\bara{\varphi\tensor{\tend{\Ee}}p}\big)&=&  
\zeta\big(\bara{\varphi\tensor{\tend{\Ee}}p}\big)(gf) \\ &=&
\zeta\Big(\big(\bara{\varphi\tensor{\tend{\Ee}}p}\big)_1\Big)(g)
\, \zeta\Big(\big(\bara{\varphi\tensor{\tend{\Ee}}p}\big)_2\Big)(f)  \\ &=& \sum_i  \zeta\big(\bara{\varphi\tensor{\tend{\Ee}}s_i} \big)(g) \,  \zeta\big(\bara{s_i^*\tensor{\tend{\Ee}}p} \big)(f)\\ &=&
\sum_i  \B{\Theta}_{\gG_1}(g)\big(\bara{\varphi\tensor{\tend{\Ee}}s_i} \big) \,  \B{\Theta}_{\gG_1}(f)\big(\bara{s_i^*\tensor{\tend{\Ee}}p} \big) \\ &=& \big(\B{\Theta}_{\gG_1}(g) \B{\Theta}_{\gG_1}(f)\big)\lr{\bara{\varphi\tensor{\tend{\Ee}}p}}
\end{eqnarray*}
where the last equality follows from the definition of the functor $\chi_{\Bbbk}$ and the comultiplication of $\Repf{G}$.
This shows that $\B{\Theta}_{\gG_1}(g) \B{\Theta}_{\gG_1}(f) = \B{\Theta}_{\gG_1}(gf)$ and finishes the proof of item $(i)$.

$(ii)$. Let $\phi:\gG \to \hH$ be  a morphism of groupoids. We need then to check the  following two equalities:
$$ \B{\Theta}_{\hH_0} \circ \phi_0\,=\, \big(\chara \circ \RepF{\phi}\big)_0  \circ \B{\Theta}_{\gG_0}, \quad   \big(\chara \circ \RepF{\phi}\big)_1  \circ \B{\Theta}_{\gG_1} \,=\,  \B{\Theta}_{\hH_1} \circ \phi_1.$$ The first equality is easily deduced from the definition of the involved maps. The second one  uses the commutative diagram of Proposition \ref{prop:R}. 
\end{proof}

Consider now  a geometrically transitive Hopf algebroid $(R, \Hh)$ and denote by $\fk{ev}: R \to \algb{\chara(R,\Hh)} ={\rm M}_{\Bbbk}(\Alg{R}{\Bbbk})$ the evaluation algebra map. We know from  \cite[Théorème 7.1, or Proposition 5.8]{Bruguieres:1994} that $\Hh$ can be reconstructed from  its category $\mathsf{comod}_{\Hh}$ of $\Hh$-comodules which are  finitely generated as $R$-modules. On the other hand, by \cite[Corollaire 3.9,  item (b) in page 114]{Deligne:1990}, this category consist in fact of those comodule which are finitely generated and projective of constant rank over $R$, since we are assuming that $\Alg{R}{\Bbbk}\neq \emptyset$.  Taking this observation into account, next we want to construct a functor from $\mathsf{comod}_{\Hh}$ to the category of representations of the characters groupoid of  $(R, \Hh)$. Let us denote by $\Uu_{\Hh}: \mathsf{comod}_{\Hh} \to \Sf{proj}(R)$ the forgetfull functor.

\begin{lemma}\label{lema:F}
Let $(R, \Hh)$ be a geometrically transitive Hopf algebroid. Then there is a functor $\fF: \mathsf{comod}_{\Hh} \to \REP{\chara(R,\Hh)}$ which turns commutative the following diagram
$$
\xymatrix@C=60pt{ \mathsf{comod}_{\Hh}  \ar@{->}^-{\fF}[rr] \ar@{->}_-{\Uu_{\Hh}}[d] & & \REP{\chara(R,\Hh)}  \ar@{->}^-{\omega}[d] \\ \Sf{proj}(R) \ar@{->}_-{\fk{ev}^*}[rr] & & \proj{\chara(R,\Hh)}}
$$
\end{lemma}
\begin{proof}
Let $P \in \mathsf{comod}_{\Hh}$ with $\Hh$-coaction $ P \ni p \mapsto p_0\tensor{R}p_1 \in P\tensor{R}\Hh$ (summation understood).  Consider  $P_x=P\tensor{R}\Bbbk_x$, for any algebra map $x \in \Alg{R}{\Bbbk}$.  Each of the $P_x$'s  is a finite dimensional vector $\Bbbk$-space with dimension the rank of $P$. We thus obtain a vector bundle $\Ee(P)=\cup_{x \in \Alg{R}{\Bbbk}}P_x$  over $\Alg{R}{\Bbbk}=\chara(R,\Hh)_0$. Now we want to endow this bundle within an   action of the groupoid $\chara(R,\Hh)$. Take an algebra map $g \in \Alg{\Hh}{\Bbbk}$, we  define the following map
\begin{equation}\label{Eq:EP} 
\varrho^{\Ee(P)}_g:  P\tensor{R}\Bbbk_{\Sf{s}^*(g)} \longrightarrow P\tensor{R}\Bbbk_{\Sf{t}^*(g)},\quad \lr{ p\tensor{R}k \longmapsto p_0\tensor{R}g(p_1)k  }
\end{equation}
where $\Sf{s}, \Sf{t}$ are the source and the target of $(R,\Hh)$. Since  the identity arrow of any $x \in \Alg{R}{\Bbbk}$ is given by $\varepsilon^*(x)$, we have 
$$ \varrho^{\Ee(P)}_{\varepsilon^*(x)}(p\tensor{R}k)\,=\, p_0\tensor{R}\varepsilon^*(x)(p_1)k\,=\,  p_0\tensor{R}x(\varepsilon(p_1))k\,=\, p_0\varepsilon(p_1)\tensor{R}k\,=\, p\tensor{R}k.$$
Hence $\varrho^{\Ee(P)}_{\varepsilon^*(x)}=id_{P_x}$, for every $x \in \Alg{R}{\Bbbk}$. On the other hand, if we take two arrows $g,f \in \Alg{\Hh}{\Bbbk}$ with $\Sf{t}^*(f)=\Sf{s}^*(g)$, then we have the following equalities\footnote{This is essentially the proof presented by diagrams in \cite[p.1299]{Hovey:2002}}
$$
\varrho^{\Ee(P)}_g \circ \varrho^{\Ee(P)}_f (p\tensor{R}k) = \varrho^{\Ee(P)}_g\Big( p_0\tensor{R}f(p_1)k \Big) = p_0\tensor{R}g(p_1)f(p_2)k = p_0\tensor{R}gf(p_1)k = \varrho^{\Ee(P)}_{gf}(p\tensor{R}k),
$$
which means that $\varrho^{\Ee(P)}_g \circ \varrho^{\Ee(P)}_f=\varrho^{\Ee(P)}_{gf}$. We have then  show that $\varrho^{\Ee(P)}$ satisfies the cocycle condition, and also it is an  isomorphism at each $g \in \Alg{\Hh}{\Bbbk}$. Therefore, $\varrho^{\Ee(P)}$ is an $\chara(R,\Hh)$-representation. This gives the definition of the functor $\fF$ on objects. 

Now it is easily seen, using the definition of the action given in \eqref{Eq:EP}, that any morphism $P \to Q$ in $\Sf{comod}_{\Hh}$ gives a rise to  a morphism $\Ee(P) \to \Ee(Q)$ between the associated $\chara(R,\Hh)$-representations, and the functor $\fF$ is now established. Lastly, the stated diagram is commutative since we know that, for every comodule $P \in \Sf{comod}_{\Hh}$, the module of global section of $\Ee(P)$ can be identified, as $\algb{\chara(R,\Hh)}$-module,  with the tensor product $P\tensor{R}\algb{\chara(R,\Hh)}$.
\end{proof}

\begin{proposition}\label{prop:counidad}
Let $(R,\Hh)$ be a geometrically transitive Hopf algebroid. 
Then there is a morphism of Hopf algebroids
$\B{\Omega}_{(R,\Hh)}:  (R,\Hh) \longrightarrow \big(\algb{\chara(R,\Hh)}, \rR_{\Bbbk} \circ \chara(R,\Hh)\big)$. Furthermore, 
$\B{\Omega}_{-}: id_{\GTCHAlgd_{\Bbbk}} \longrightarrow \rR_{\Bbbk} \circ \chara$ is a natural transformation.
\end{proposition}
\begin{proof}
Since $(R, \Hh)$ is geometrically transitive we know from \cite[Théorème 7.1]{Bruguieres:1994} that the pair $(\Sf{comod}_{\Hh}, \Uu_{\Hh})$ is a Tannakian $\Bbbk$-linear category and that the canonical morphism of Hopf algebroides $(R,\lL_{\Bbbk}(\Uu_{\Hh})) \to (R,\Hh)$ 
is an isomorphism, where $\lL_{\Bbbk}(\Uu_{\Hh})$ is the universal object attached to $(\Sf{comod}_{\Hh},\Uu_{\Hh})$, see the Appendix. 
Using the functor $\fF_{\Hh}$ constructed in  Lemma \ref{lema:F}, we know by Lemma \ref{lema:LK} that  there is  a morphism   $$\lL_ {\Bbbk}(\fF_{\Hh}): (R,\lL_{\Bbbk}(\Uu_{\Hh})) \longrightarrow \big(\algb{\chara(R,\Hh)},\rR_{\Bbbk}\chara(R,\Hh)\big)$$ of Hopf algebroids with  base change map $\fk{ev}: R \to  \algb{\chara(R,\Hh)}={\rm M}_{\Bbbk}(\Alg{R}{\Bbbk})$ the evaluation algebra map.
Therefore, we have a composition of  morphism 
\begin{equation}\label{Eq.Omega}
\B{\Omega}_{(R,\Hh)}: (R, \Hh) \cong (R,\lL_{\Bbbk}(\Uu_{\Hh})) \longrightarrow \big(\algb{\chara(R,\Hh)},\rR_{\Bbbk}\chara(R,\Hh)\big)
\end{equation}
of Hopf algebroids as it was stated. 

To show that $\B{\Omega}_{-}$ is natural, one need to check the following two equalities 
$$
\big(\rR_{\Bbbk}\circ \chara(\alpha)\big)_0 \circ \B{\Omega}_{(R,\Hh)_0} \,=\,  \B{\Omega}_{(S,\Kk)_0} \circ \alpha_0, \quad 
\B{\Omega}_{(S,\Kk)_1} \circ \alpha_1\,=\, \big(\rR_{\Bbbk}\circ \chara(\alpha)\big)_1 \circ \B{\Omega}_{(R,\Hh)_1} 
$$
for any morphism  $\alpha=(\alpha_0,\alpha_1): (R,\Hh) \to (S,\Kk)$ between geometrically transitive Hopf algebroids. The first equality is obviously obtained from the definitions. While the proof of the second is given as follows. Since  $\alpha$ is  a morphism of  Hopf algebroids, we have  the following commutative diagram
$$
\xymatrix@C=35pt@R=25pt{  &  &&  \scriptstyle{\Sf{comod}_{\Kk}} \ar@{->}|-{\scriptstyle{\Uu_{\Kk}}}[dd]|\hole  \ar@{->}^-{\scriptstyle{\fF_{\Kk}}}[r] & \scriptstyle{\REP{\chara(S,\Kk)}}  \ar@{->}|-{\scriptstyle{\omega_{\chara(S,\Kk)}}}[dd] \\ 
\scriptstyle{ \Sf{comod}_{\Hh}} \ar@{->}|-{\scriptstyle{-\tensor{R}S}}[rrru] \ar@{->}|-{\scriptstyle{\Uu_{\Hh}}}[dd] \ar@{->}_-{\scriptstyle{\fF_{\Hh}}}[r]  & \scriptstyle{\REP{\chara(R,\Hh)}} \ar@{->}|-{\scriptstyle{\Rr(\chara(\alpha))}}[rrru] \ar@{->}|-{\scriptstyle{\omega_{\chara(R,\Hh)}}}[dd]|\hole  &&   &  \\ 
&  && \scriptstyle{\Sf{proj}(S}) \ar@{->}^-{\scriptstyle{\fk{ev}^*}}[r]  & \scriptstyle{\Sf{proj}(\algb{\chara(S,\Kk)})}  \\ 
\scriptstyle{\Sf{proj}(R)} \ar@{->}|-{\scriptstyle{\alpha_0^*}}[rrru] \ar@{->}_-{\scriptstyle{\fk{ev}^*}}[r]  &\scriptstyle{ \Sf{proj}(\algb{\chara(R,\Hh)})} \ar@{->}|-{\scriptstyle{{\rm M}_{\Bbbk}({\Alg{\alpha_0}{\Bbbk})^*}}}[rrru]  &&    & 
}
$$
which implies the equality  $\fF_{\Kk} \circ (-\tensor{R}S) \,=\, 
\Rr(\chara(\alpha)) \circ \fF_{\Hh}$ in the category $\Tanna$ (see the Appendix), where $\mathscr{F}_{-}$ are the functors defined in Lemma \ref{lema:F}.  Now applying the functor $\lL_{\Bbbk}$ to this equality gives the desired equation. 
\end{proof}

\subsection{The main Theorem.}\label{ssec:thm}
Now we dispose of all  ingredients to state our main theorem. We will show that the contravariant functors $\chara$ and $\rR_{\Bbbk}$ establish  a duality between geometrically transitive Hopf algebroids and discrete groupoids. 
\smallskip

Let $\gG$ be a groupoid and $B=\algb{\gG}$ its base $\Bbbk$-algebra. There are various evaluating maps which will be used. So a distinguishing notation is helpful. These are 
$$
ev: \gG_0 \to \Alg{B}{\Bbbk},\quad  \fk{ev}: B \to {\rm M}_{\Bbbk}(\Alg{B}{\Bbbk}),\quad \Sf{ev}: {\rm M}_{\Bbbk}(\Alg{B}{\Bbbk}) \to \algb{\gG}=B
$$
where the last algebra map uses the first one and we have $\Sf{ev} \circ \fk{ev}=id_B$.

\begin{lemma}\label{lema:triangulo1}
Consider the natural transformations  $\B{\Theta}$ and $\B{\Omega}$, respectively, of Lemma \ref{lema:unit}(ii) and Proposition \ref{prop:counidad}. Then for any groupoid $\gG$ with base $\Bbbk$-algebra $B$, we have 
$$
\rR_{\Bbbk}(\B{\Theta}_{\gG}) \circ \B{\Omega}_{(B, \Repf{G})}\,=\, id_{(B,\Repf{G})}.
$$ 
\end{lemma}
\begin{proof}
We know that $(B, \Repf{G})$ is geometrically transitive. In particular,  we have a monoidal equivalence of categories $\Sf{comod}_{\Repf{G}}\cong \Rep{G}$ (in fact an isomorphism in the category $\Tanna$). In one direction the functor $\Rep{G} \to \Sf{comod}_{\Repf{G}}$  takes any $\gG$-representation $\varrho^{\Ee}$, to its global sections module $\Gama{\Ee}$ with the following  $\Repf{G}$-comodule  structure  
$$
\Gama{\Ee} \longrightarrow \Gama{\Ee}\tensor{B}\Repf{G}, \quad \Big(s \longmapsto \sum_i s_i \tensor{B} (\bara{s_i^*\tensor{\tend{\Ee}}s})\Big)
$$
where $\{s_i,s_i^*\}$ is the dual basis of the module $\Gama{\Ee}$. In this way, we compose this functor  with the functor $\fF$ of Lemma \ref{lema:F} to obtain a new functor $\fF_{\Repf{G}}: \Rep{G} \to 
\REP{\chara(B,\Repf{G})}$. 
We then arrive to  the following commutative diagram 
$$
\xymatrix@C=50pt{ & \scriptstyle{\REP{\chara(B,\Repf{G})}}  \ar@{->}^-{\scriptstyle{\Rr(\B{\Theta}_{\gG})}}[dr] \ar@{->}|-{\scriptstyle{\omega_{\chara(B,\Repf{\gG})}}}[ddd]|\hole & \\ \scriptstyle{\Rep{G}} \ar@{->}_-{\scriptstyle{\omega_{\gG}}}[ddd]  \ar@{->}^-{\scriptstyle{\fF_{\Repf{G}}}}[ru]  \ar@{-}^>>>>>>>>>{\cong}[rr]|\hole & & \scriptstyle{\Rep{G} \ar@{->}^-{\scriptstyle{\omega_{\gG}}}[ddd]} \\ & & \\ & \scriptstyle{\Sf{proj}({\rm M}_{\Bbbk}(\Alg{B}{\Bbbk})} \ar@{->}^-{\scriptstyle{\Sf{ev}^*}}[rd]  & \\  \scriptstyle{\Sf{proj}(B)} \ar@{=}[rr]  \ar@{->}^-{\scriptstyle{\fk{ev}^*}}[ru] & &  \scriptstyle{\Sf{proj}(B)}} 
$$
whose proof is based on  the fact that for any $\gG$-representation $\varrho^{\Ee}$  the fibers of the underlying bundle of the $\gG$-representation $\Rr(\B{\Theta}_{\gG}) \circ \fF_{\Repf{\gG}}(\varrho^{\Ee})$ are identified with those of $\Ee$, as well as their  $\gG$-actions. Now the stated equality follows by applying the functor $\lL_{\Bbbk}$ (see Appendix) to the above diagram. 
\end{proof}

\begin{lemma}\label{lema:triangulo2}
Consider the natural transformations  $\B{\Theta}$ and $\B{\Omega}$, respectively, of Lemma \ref{lema:unit}(ii) and Proposition \ref{prop:counidad}. Then for any geometrically transitive Hopf algebroid $(R,\Hh)$, we have 
$$
\chara\big({\B{\Omega}}_{(R,\Hh)} \big)  \circ \B{\Theta}_{\chara(R,\Hh)}\,=\, id_{\chara(R,\Hh)}.
$$ 
\end{lemma}
\begin{proof}
We need to check the following two equalities
\begin{equation}\label{Eq:triang2}
\chara\big({\B{\Omega}}_{(R,\Hh)} \big)_0  \circ \B{\Theta}_{\chara(R,\Hh)_0}\,=\, id_{\Alg{R}{\Bbbk}}\;\; \text{ and } \quad \chara\big({\B{\Omega}}_{(R,\Hh)} \big)_1  \circ \B{\Theta}_{\chara(R,\Hh)_1}\,=\, id_{\Alg{\Hh}{\Bbbk}},
\end{equation}
The right hand term of the first one is just the composition of the following two maps
$$ 
\xymatrix@R=0pt{ \Alg{R}{\Bbbk} \ar@{->}[r] & \Alg{{\rm M}_{\Bbbk}(\Alg{R}{\Bbbk})}{\Bbbk} \\ x  \ar@{|->}[r] & \Lr{\beta \mapsto \beta(x)}  } \qquad 
\xymatrix@R=0pt{  \Alg{{\rm M}_{\Bbbk}(\Alg{R}{\Bbbk})}{\Bbbk} \ar@{->}[r] & \Alg{R}{\Bbbk}  \\ \delta \ar@{|->}[r] & \Lr{r \mapsto \delta(ev_r)}  }
$$ where $ev: R \to {\rm M}_{\Bbbk}(\Alg{R}{\Bbbk})$ sends $r \mapsto \big[x \mapsto x(r) \big]$, which is the identity of $\Alg{R}{\Bbbk}$. 

Now we want to check the second equality in \eqref{Eq:triang2}. To this end, we first identify $(R, \Hh)$ with the Hopf algebroid $(R, \lL_{\Bbbk}(\frcomod{\Hh}))$. Under this  identification the coaction of $Q \in \frcomod{\Hh}$ is given by    $q \mapsto q_0\tensor{R}q_1= q_i\tensor{R} (\bara{q_i^*\tensor{{\rm T}_{Q}}q})$ for some fixed dual basis $\{q_i,q_i^*\}$ of $Q_R$. On the other hand the elements of $\Hh$ are now considered as sum of generic elements of the form 
$\bara{\psi\tensor{{\rm T}_{Q}}q} \in  \Hh$ for some $\Hh$-comodule $Q \in \frcomod{\Hh}$ with $\psi \in Q^*, \, q \in Q$. 
In this way the left hand-side term of the second equality in equation \eqref{Eq:triang2} is explicitly given by the map
\begin{equation}\label{Eq:EF}
\Alg{\Hh}{\Bbbk} \longrightarrow   \Alg{\Hh}{\Bbbk}, \quad \lr{g \longmapsto \Big[ \bara{\psi\tensor{{\rm T}_{Q}}q} \mapsto (\psi\tensor{R}1)(\Sf{t}^*(g)) \varrho^{\Ee(Q)}_g\big((q\tensor{R}1)(\Sf{s}^*(g)\big) \Big]}
\end{equation}
where the bundle $\Ee(Q)$ is $\cup_{x \, \in \Alg{R}{\Bbbk}}Q\tensor{R}\Bbbk_x$ and its action $\varrho^{\Ee(Q)}_g$ is as in equation \eqref{Eq:EP}. Here $\Sf{s}$  and $\Sf{t}$  are, respectively, the source and the target of $(R,\Hh)$. Computing the resulting value in \eqref{Eq:EF}, we find that  
\begin{eqnarray*}
(\psi\tensor{R}1)(\Sf{t}^*(g)) \varrho^{\Ee(Q)}_g\big((q\tensor{R}1)(\Sf{s}^*(g)\big) &=& (\psi\tensor{R}1)(\Sf{t}^*(g)) \big(q_0\tensor{R}g(q_1)\big) \\ &=& ((g \circ \Sf{t} \circ \psi )\tensor{R}1) \big(q_0\tensor{R}g(q_1)\big) \\ &=& g(\Sf{t}(\psi(q_0))g(q_1) \\ &=& g\Big( \Sf{t}(\psi(q_0))q_1 \Big) \\ &=& g\Big( \Sf{t}(\psi(q_i)\, \bara{q_i^*\tensor{R}q} \Big) \\ &=& g\Big(\bara{( \Sf{t}(\psi(q_i)q_i^*)\tensor{R}q} \Big) \\ &=& g\Big(\bara{\psi\tensor{R}q} \Big),
\end{eqnarray*}
which shows that the map of equation \eqref{Eq:EF} is the identity and finishes the proof. 
\end{proof}

Our main result is the following theorem.
\begin{theorem}\label{thm:main}
Let $\Bbbk$ be a field. Then the contravariant functors of $\Bbbk$-characters groupoid  and $\Bbbk$-valued representative functions 
$$\xymatrix{\chara: \Sf{GTCHAlgd}_{\Bbbk} \ar@<-0.5ex>@{->}[r] & \Grpd: \rR_{\Bbbk} \ar@<-0.5ex>@{->}[l]  }$$ 
establish a duality between the category of  geometrically transitive commutative Hopf algebroids and the category of discrete groupoids. That is, there is a natural isomorphism:
$$
{\rm Hom}_{\scriptstyle{\Sf{GTCHAlgd}_{\Bbbk}}}\lr{{(R,\Hh)}\, ;\,{\big(\algb{\gG},\Repf{G}\big)} }\,\, \cong \,\,  {\rm Hom}_{\scriptstyle{\Grpd}}\lr{{\gG}\, ; \,{\chara(R,\Hh)}}
$$
for every $(R,\Hh) \in \Sf{GTCHAlgd}_{\Bbbk}$ and $\gG \in \Grpd$.
\end{theorem}
\begin{proof}
This is a direct consequence of Lemmas \ref{lema:triangulo1} and \ref{lema:triangulo2}.
\end{proof}

\begin{remark}\label{rem:final}
In analogy with  the anti-equivalence between compact  topological  groups and commutative Hopf algebras with (positive) integral\footnote{which corresponds to the Haar measure on the group.}and  dense characters group, which was proved in  \cite[Theorem 3.5, page 30]{Hochschild:book1}, see also \cite[Theorem 3.4.3] {Abe:book}. It is the belief of the author that a similar duality as in  Theorem \ref{thm:main} by using the functor $\rR_{\Bbbk}^{\scriptscriptstyle{top}}$,  could be performed.  In this direction, we expect that it is  possible to construct an anti-equivalence of categories between the category of compact topological groupoids and a certain full subcategory of commutative Hopf algebroids. Parallel to the results of \cite{Amini:2007}, such an anti-equivalence could be taught as a Tannaka-Krein duality for compact  topological groupoids. 

Of course, similar to the classical case of compact topological groups, here there are surely two essential difficulties which one is required  to overcome. The  first one is after endowing a given compact topological groupoid $\gG$  within a  (left) Haar system, we are required to check how this can be  reflected in its Hopf algebroid $\tRepf{G}$. This perhaps is reflected  in terms of a system of integrals in the isotropy Hopf algebras. The  second difficulty, is to proof a  version of Peter-Weyl's theorem in this context, namely, that $\tRepf{G}$ is dense in  the total algebra $\calgt{\gG}$.  
\end{remark}

\appendix

\section{Basic Tannakian $\Bbbk$-linear categories and the functor $\lL_{\Bbbk}$.}\label{sec:TL}
We present here a brief account on Tannakian $\Bbbk$-linear\footnote{$\Bbbk$-linear means enriched in $\Bbbk$-vector spaces.}categories.  What is essential for our use is the construction of the functor $\lL_{\Bbbk}$ from the category of all Tannakian $\Bbbk$-linear categories to the category of geometrically transitive Hopf algebroids with ground field $\Bbbk$.

For more details on Tannakian $\Bbbk$-linear categories, we refer to \cite{Deligne:1990, Deligne/Milne, Bruguieres:1994}, see also \cite{El Kaoutit/Gomez:2004}.  
Recall from \cite[\S 2]{Deligne:1990} (see also \cite[\S 2]{Bruguieres:1994} for a weaker definition) that a symmetric  monoidal rigid (or autonomous) $\Bbbk$-linear (essentially small) category $(\Tt, \tensor{}, \I)$\footnote{We are implicitly assuming that $\tensor{}$ is a $\Bbbk$-bilinear bi-functor.}, is said to be a \emph{$\Bbbk$-tensorial category }  "est une cat\'egorie tensorielle sur $\Bbbk$", if its underlying category $\Tt$ is abelian and the canonical algebra map $\Bbbk \to \End{\Tt}{\I}$ is an isomorphism.  
Let $(\Tt, \tensor{}, \I)$ be a  $\Bbbk$-tensorial category, and  $R$ a commutative $\Bbbk$-algebra.  Following \cite[1.9]{Deligne:1990} (see also \cite[Definition p.5826]{Bruguieres:1994}), a \emph{fiber functor of $\Tt$ over $R$} is a monoidal $\Bbbk$-linear faithful and right exact functor from $\Tt$ to the category $\Sf{proj}(R)$ of finitely generated and projective $R$-modules. Such a functor $\omega: \Tt \to \Sf{proj}(R)$  satisfies $\omega(\I) \cong R$ and a natural isomorphism $\omega(-\tensor{}-)\cong \omega(-)\tensor{R}\omega(-)$ compatible with both associativity and  symmetry of $\tensor{}$. A \emph{Tannakian $\Bbbk$-linear category} \cite[2.8]{Deligne:1990}, is then a $\Bbbk$-tensorial category with a fiber functor $\omega: \Tt \to \Sf{proj}(R)$ (here in fact we are restricting the definition \cite[2.8]{Deligne:1990} to affine schemes).

Notice here that  $\omega$ is not trivial, since we are assuming that $\Alg{R}{\Bbbk}\neq \emptyset$.  In particular, we have by \cite[Proposition 2.5]{Bruguieres:1994} that $\Tt$ is \emph{locally of finite type over $\Bbbk$} (see the definition before \cite[Proposition 2.5]{Bruguieres:1994}). 
We denote the situation of a given Tannakian $\Bbbk$-linear category  by $(\Tt, \omega)_R$. Tannakians categories are objects of the category $\Tanna$ where a morphism $(\Tt,\omega)_R \to (\Pp, \gamma)_{S}$ between two Tannakian categories consists of $\Bbbk$-algebra map $\theta:R \to S$ and a functor $\Sf{f}:\Tt \to \Pp$ such that the following diagram is commutative
$$ 
\xymatrix@C=50pt{ \Tt \ar@{->}^-{\Sf{f}}[r] \ar@{->}_-{\omega}[d]  & \Pp \ar@{->}^-{\gamma}[d]  \\ \Sf{proj}(R) \ar@{->}^-{\theta^*}[r] & \Sf{proj}(S).}
$$

To each object in $(\Tt,\omega)_R \in \Tanna$ one can associate a universal object denoted $\lL_{\Bbbk}(\omega)$ which turns to be a commutative Hopf $R$-algebroid. In the notation of \cite{Deligne:1990}, this  is $\underline{\bf Aut}^{\tensor{}}(\omega)$ the set of all monoidal natural isomorphisms of $\omega$. In categorical terms, the universality of  $\lL_{\Bbbk}(\omega)$ can be expressed at least in two different (rather equivalent) ways. 
The first one says that $\lL_{\Bbbk}(\omega)$ is the $(R\tensor{\Bbbk}R)$-bimodule which  solves the following universal  problem in $R$-bimodules between natural transformations and $R$-bilinear maps:
\begin{eqnarray*}
{\rm Nat}\lr{\omega,\, -\tensor{R}\omega} &\cong & \hom{R\text{-}R}{\lL_{\Bbbk}(\omega)}{-}, \\ {\rm Nat}\Big(\omega\tensor{R}\omega,\, -\tensor{R}(\omega\tensor{R}\omega)\Big) & \cong & \hom{R\text{-}R}{\lL_{\Bbbk}(\omega)\tensor{R^{\Sf{e}}}\lL_{\Bbbk}(\omega)}{-},
\end{eqnarray*}
where $R^{\Sf{e}}=R\tensor{\Bbbk}R$ is the enveloping algebra of $R$. 

The second one says that $\lL_{\Bbbk}(\omega)$ is the $R$-coring\footnote{What is called here an $R$-coring is called "$\Bbbk$-cogébroïde de base $R$" in \cite{Deligne:1990,Bruguieres:1994}.} which represents the following functor 
$$
\Coring{R} \longrightarrow \Sf{Sets}, \quad \Big( \coring{C} \to \{ \footnotesize{\text{funtorial right } \coring{C}\text{-coactions on } \omega }\}  \Big)
$$ form the category of $R$-corings to sets, see \cite[Proposition 4.2]{Bruguieres:1994} for more details.

To our need a concert description of $\lL_{\Bbbk}(\omega)$ is primordial.  As it is known by specialist  such a description  is not trivial and  rather quite technical. Next, we propose three descriptions of $\lL_{\Bbbk}(\omega)$ two of them use the so called reconstruction process which was  developed in  \cite{El Kaoutit/Gomez:2004}.

Let us first  fix some notations.  
For sake of simplicity, the vector spaces of morphisms in $\Tt$ will be denoted by 
$$ \thom{X}{Y}:= \hom{\Tt}{X}{Y},\quad 
\tend{X}:= \mathrm{End}_{\Tt} (\varrho^{X}).$$
In this way we can perform the direct sum of $R$-bimodules 
\begin{equation}\label{Eq:JG}
\bigoplus_{X \,\in\, \Tt} \omega(X)^*\tensor{\tend{X}}\omega(X), \text{ and  } \jJ_{\Tt}\,=\, \Big\langle \underset{}{}  \varphi \tensor{\tend{Y}}\alpha s - \varphi \alpha\tensor{\tend{X}}s \Big \rangle_{\varphi \in \omega(Y)^*,\, s \in \omega(X),\, \alpha \in \thom{X}{Y}},
\end{equation}
its $(R\tensor{\Bbbk}R)$-submodule generated by this set and  where a certain canonical pairings were used. 

On the other hand we set 
$$\Bb:=\bigoplus_{X,\, Y\, \in \,\Tt} \mathrm{T}_{X,\, Y}\,=\, \bigoplus_{X,\, Y\, \in \,\Tt} \hom{\Tt}{X}{Y}
$$ 
the Gabriel's ring\footnote{This is a ring with local units, namely, with  enough orthogonal idempotents the identities arrows of $\Tt$.}associated to the essentially small $\Bbbk$-linear category $\Tt$, together with the following unital bimodules:
\begin{equation}\label{Eq:S}
\B{\Sigma}:= \bigoplus_{X \,\in\, \Tt} \omega(X),\text{ and } 
\B{\Sigma}^ {\dag}:= \bigoplus_{X\,\in\, \Tt}\omega(X)^*,
\end{equation}
where $\B{\Sigma}$ is an unital $(\Bb,R)$-bimodule while $\B{\Sigma}^{\dag}$ is an unital $(R,\Bb)$-bimodule. 

In summary, the universal $R$-bimodule $\lL_{\Bbbk}(\omega)$ which soloves the above problems, is given by 
\begin{equation}\label{Eq:Lomega}
\lL_{\Bbbk}(\omega) \,\,\cong\,\, \int^{X\,\in\, \Tt} \omega(X)^*\tensor{\Bbbk}\omega(X)\footnote{This is the coend object of the functor $\omega(-)^*\tensor{\Bbbk}\omega(-)$ in $R$-bimodules, see \cite{MacLane-Categories}.} \, \cong\, \frac{\bigoplus_{X\, \in \,\Tt} \omega(X^*)\tensor{\tend{X}}\omega(X)}{\jJ_{\Tt}} \,\, \cong \,\, \B{\Sigma}^{\dag}\tensor{\Bb} \B{\Sigma}.
\end{equation}
In this way we arrive to the following lemma which is extremely useful in this context. 
\begin{lemma}\label{lema:LK}
The universal object of equation \eqref{Eq:Lomega}  establishes a well defined covariant functor $$\lL_{\Bbbk}: \Tanna \longrightarrow \Sf{GTCHAlgd}_{\Bbbk}$$ to the category of geometrically transitive and commutative  Hopf  algebroids. 
\end{lemma}
\begin{proof}
This is a consequence of \cite[Théorèmes 5.2,  7.1 et 8.2]{Bruguieres:1994}.
\end{proof}
As a matter of indication we recall that the functor $\lL_{\Bbbk}$ is elementwise defined as follows. We use the second description of equation \eqref{Eq:Lomega}, that is, an element in $\lL_{\Bbbk}(\omega)$ is a sum of generic elements of the form $\bara{\varphi\tensor{{\rm T}_X}\Sf{x}}$  which is the  equivalence class of the element $\varphi\tensor{{\rm T}_X}\Sf{x} \in \omega(X)^*\tensor{{\rm T}_X}\omega(X)$. Consider a morphism  $(\theta, \Sf{f}): (\Tt,\omega)_R \to (\Pp, \gamma)_{S}$ in $\Tanna$, then the corresponding morphism of Hopf algebroids is given by 
$$
(\theta, \lL_{\Bbbk}(\Sf{f})): (R, \lL_{\Bbbk}(\omega)) \longrightarrow (S, \lL_{\Bbbk}(\gamma)), \quad \lr{\Big(r, \bara{\varphi\tensor{{\rm T}_X}\Sf{x}}\Big) \longmapsto \Big(\theta(r), \bara{(\varphi\tensor{R}1)\tensor{{\rm T}_{\Sf{f}(X)}}(\Sf{x}\tensor{R}1)} \,\Big) }$$
where 
$\varphi\tensor{R}1 \, \in \gamma(\Sf{f}(X))=\omega(X)^*\tensor{R}S$ is defined by sending $u\tensor{R}s \mapsto \varphi(u)s$.

\begin{remark}\label{rem:GT}
As one can realize the category $\Tanna$ is in fact a $2$-category with $0$-cells  are segments, $1$-cells are squares and $2$-cells are cubs. It is also  possible to endow $\Sf{GTCHAlgd}_{\Bbbk}$ within a structure of bicategory in such a way that $\lL_{\Bbbk}$ becomes a (strict) homomorphism of bicategories. At this level, the search of a $2$-adjoint  of $\lL_{\Bbbk}$ could be of extremely interest in this framework.

On the other hand, the construction of the functor $\lL_{\Bbbk}$ can be performed in  the more general  case of non necessarily abelian $\Bbbk$-linear categories. Precisely, we can consider the category whose objects are three-tuples of the form $(\Aa, \omega)_A$ where  $\Aa$ is a $\Bbbk$-linear symmetric and rigid monoidal (essentially small) category, $\omega: \Aa \to \Sf{proj}(A)$ is a monoidal faithful functor, and  $A$ is a commutative $\Bbbk$-algebra. The morphisms in this category are as above. 

In this way,  we can construct exactly as in equation \eqref{Eq:Lomega} a functor $\bara{\lL}_{\Bbbk}$ form this category to the category of commutative Hopf algebroids. For example,  the algebra of continuous representative functions  $\tRepf{\gG}$ of a topological groupoid $\gG$ with compact Hausdorff base space, described in  Remark \ref{rem:RFTops}, is  the image of the object $(\tRep{\gG}, \omega^{\scriptscriptstyle{top}})_{\calgb{\gG}}$ by this functor  $\bara{\lL}_{\Bbbk}$.
In general, some of the properties of $\bara{\lL}_{\Bbbk}(\omega)$, the image of an object  $(\Aa, \omega)_A$, can be deduced from that of the   unital bimodule ${}_{\Bb}\B{\Sigma}_A$ described in equation \eqref{Eq:S}.
For instance, one can show by using the third description  of equation \eqref{Eq:RG}, that   $\bara{\lL}_{\Bbbk}(\omega)$ is projective as an $(A\tensor{\Bbbk}A)$-module,  if the  unital right $(\Bb\tensor{\Bbbk}A)$-module\footnote{Here $\Bb\tensor{\Bbbk}A$ is considered as a ring with enough orthogonal idempotents, namely, the set $\{id_{X}\tensor{\Bbbk}1_A\}_{X \,\in \Aa}$, where $\Bb$ is the Gabriel ring of $\Aa$.} $\B{\Sigma}$ is projective. 
\end{remark}

\smallskip

\textbf{Acknowledgements.} I  would like to thank Alain Bruguières for carefully reading the first version of the draft.   I would like to thank also  Fabio Gavarini and Niels Kowalzig for helpful discussions about  this subject, for encourage me to write  this note and  for inviting me to visit the  Dipartimento di Matematica, Universit\`a di Roma "Tor Vergata".

\end{document}